\documentclass[12pt]{article}
\usepackage{pudlatex}
\usepackage{fullpage}
\usepackage[all,cmtip]{xy}

\newtheorem{que}{Question}

\newenvironment{conjecture}[1]{\par\bigskip\noindent{\bf Conjecture }{{\bf (}#1{\bf )}}\begin{it}}{\end{it}\bigskip}

\newcommand{\tto}{\Rightarrow}
\newcommand{\dok}{\vdash^*}%{\stackrel{*}{\vdash}}

\title{Incompleteness in the finite domain}

\author{Pavel Pudl\'ak
\thanks{The author is supported by the ERC Advanced Grant
  339691 (FEALORA) and the institute grant RVO: 67985840. Part of this article was written while the author was visiting Simons Insitute in Berkeley, California.}}

\begin{document}
\maketitle

\begin{abstract}
Motivated by the problem of finding finite versions of classical incompleteness theorems, we present some conjectures that go beyond ${\bf NP\neq
  co NP}$. These conjectures formally connect computational complexity with the difficulty of proving some sentences, which means that high computational complexity of a problem associated with a sentence implies that the sentence is not provable in a weak theory, or requires a long proof. Another reason for putting forward these
conjectures is that some results in proof complexity seem to be special cases of such general statements and we want to formalize and fully understand these statements. In this paper we review some conjectures that we have presented earlier~\cite{KP,godel100,kniha,herbr}, introduce new conjectures, systematize them and prove new connections between them and some other statements studied before. 

\end{abstract}

\section{Introduction}

G\"odel's incompleteness theorem is undoubtedly one of the most
important theorems in logic. 
%This surprising result changed the
%way people were viewing the foundations of mathematics. 
It speaks about absolute provability, i.e., about
proofs without any restriction on their length. The question whether
there is a ``finite'' or ``feasible'' version of the incompleteness
theorem, where the complexity of proofs is bounded, has certainly intrigued many people, but very little has been 
published about it. With the advent of computers and theories
developed for them, in particular complexity theory, the question
about a finite version of the incompleteness theorem became even more
interesting. The concept of polynomial time computations turned out to
be the most important concept in complexity theory. The distinction between functions decidable in polynomial time and those computable only in exponential time plays a similar role as the distinction between computable and non-computable in computability theory. 
The successful use
of polynomial bounds suggested that one should also study which
theorems have polynomial length proofs. A natural version of
a finite incompleteness theorem was formulated by Harvey Friedman in~1979. Let
$Con_T(\bar n)$ be a natural formalization of the statement
\emph{``there is no derivation of contradiction of length $n$ from the
  axioms of $T$''}. Friedman proved a lower bound of the form $n^\epsilon$ for some $\epsilon>0$ and asked whether such sentences have proofs in
$T$ of polynomial length~\cite{friedman79}.%
\footnote{Here $\bar n$ denotes a \emph{binary numeral}, a term of length $O(\log n)$ that represents $n$.}
It turned out that the answer to his question is
yes~\cite{pudlak86}, but it is still possible, and seems very plausible, that for natural variations of this question
there are no polynomial length proofs. Namely, this should be true if
we ask about the lengths of proofs of $Con_T(\bar n)$ in a theory $S$ sufficiently
\emph{weaker} than~$T$. However, proving such a claim must be
extremely difficult, because it implies ${\bf P\neq NP}$ (and even more
than that).
% We have to face the fact that present-day mathematics lacks methods to prove ${\bf P\neq NP}$, and hence also to prove any conjecture that implies it, but there are a lot of other things that we can do. 

Our motivation for studying such problems is the fundamental question: \emph{what is the connection
  between logical strength of theories and computational complexity?}
which is basically what the field of \emph{proof complexity} is
about. Here we refer to proof complexity in a broader sense that also includes the study of first order theories called bounded arithmetic. Since there is a close connection between propositional proof systems and first order theories, we view these two concepts as nonuniform and uniform versions of the same concept.

%{We advocate the notion that bounded arithmetic is a part of
%proof complexity.} %Eng OK

To give an example of a connection between theories and computational complexity, let us consider Buss's Witnessing Theorem~\cite{buss86}. This theorem states that one can
construct polynomial time algorithms from proofs of certain sentences
in the theory $S^1_2$ (see Theorem~\ref{t-witnessing} below). This is an instance of a general phenomenon: if a theory is weak, the provably total functions have small computational complexity. 
Such theorems have been proven
for a number of other theories and complexity classes. Another
connection is the Feasible Interpolation Theorem of
Kraj\'{\i}\v{c}ek~\cite{krajicek97}. According to this theorem, one
can construct circuits from proofs of certain tautologies in various
proof systems, in particular, in resolution; the circuits separate two sets of Boolean vectors defined by the tautology.
%(Such theorems have beenproven also for other proof systems.)  
A high level form of these
results is that if something is provable in a weak formal system,
i.e., the logical strength of the system is bounded, we can give
bounds on some computational problems associated with the systems. If
we state it contrapositively it suggests that increasing strength of
logical formal systems is correlated with increasing complexity of the
associated computational tasks. Thus a more specific question is:
\emph{find general principles of which these results are special
  instances.}

Connection between
proofs and computations have been extensively studied in constructive mathematics in the context of intutionistic logic. There are also results that show interesting
connections of proofs in the intutionistic calculus with computational complexity. For example, Buss and
Mints~\cite{BussMints} proved that given an intuitionistic proof of a
disjunction $\phi\vee\psi$ in propositional logic (say, in the sequent
calculus), one can find \emph{in polynomial time} a proof of either
$\phi$ or $\psi$. A more general theorem, a version of a realizability theorem for intuitionistic propositional calculus, was proved in~\cite{buss-p}. 
However, the problems we are going to consider in this paper have not been studied in the context of intuitionistic logic and also in this paper we will only use classical logic.

%% . The reason is that in what we study here, the underlying
%% logic is not important. What is essential are proof systems and axioms
%% used. Considering intuitionistic logic would only be a restriction on
%% the systems we could use, while we are interested in formal
%% systems that are as general as possible.

The general principles that we study are connected with notoriously
open and probably very difficult problems in computational complexity
theory, so we cannot prove or disprove them with the currently
available means. They can only be stated as hypotheses or conjectures
without any formal supporting evidence. 
There are, essentially, two
reasons for stating some sentences as conjectures.
First, we believe that some basic theorems of proof theory should also hold true with suitable bounds on the lengths of proofs. The prime example is
the Second Incompleteness Theorem discussed above.
Second, some results in proof complexity and bounded arithmetic seem to follow a general pattern.
For example, as we noted above, polynomial time computations are
associated with the theory $S^1_2$ by a witnessing theorem. If we take
$S^2_2$, which we believe is a stronger theory, then the corresponding
function class is ${\bf P^{NP}}$,%
\footnote{This was also proved by Buss in~\cite{buss86}.} 
which we believe is a larger class
than {\bf P}.% 
\footnote{We are not able to prove it formally, because a 
formal proof would give us ${\bf P^{NP}\neq P}$, which is equivalent
to ${\bf NP\neq P}$.}  
The form of this result suggests that $S^2_2$ requires
more complex functions.

Alternatively, we can view the proposed conjectures as axioms. In fact, ${\bf NP\neq P}$ has been treated as an axiom ``to prove'' hardness of various problems. In many cases ${\bf NP\neq P}$ does not suffice and therefore a number of stronger hypotheses have been proposed. For example, the \emph{Exponential Time Hypothesis} of  Impagliazzo and Paturi~\cite{IP} and its variants are used to determine the time complexity of concrete polynomial time computable problems. In the theory of approximate algorithm several conjectures have been proposed in order to show nonapproximability of certain problems.

Although we have to treat the most interesting statements only as hypotheses, there are some interesting problems that we can study and
solve with the currently available means. These are problems about
relationships among various conjectures.  In particular, we would like
to know whether there is one general principle that would cover all
instances, or there is an infinite hierarchy. If there is a hierarchy, is it
linear, or does it branch? If it branches, is there a natural
classification of conjectures? We will address some questions of this
kind in this paper. Furthermore, one can study relativizations of
these conjectures. Several results about relativizations have been proven,
but much more is needed. We will mention some more concrete problems at the and of the paper.

We are primarily interested in these questions, because we want to
understand the essence of fundamental problems. 
%% enables us to view open problems in complexity theory from a different
%% perspective. 
However, there is also a practical aspect of this
research. The general conjectures suggest what specific problems in
proof complexity we should study. Then we can ``test'' the
conjectures on weak formal systems for which we do have means to prove
results connecting them with computational complexity. In fact the main Conjectures~{\sf CON}  and~{\sf TFNP} represent what researchers in proof complexity believe is likely to be true. %% Therefore it may be useful to collect these hypotheses and find out how they are related. 

 All conjectures that we consider in this paper state something about
 unprovability, although they often have a natural equivalent
 version stated in purely complexity-theoretical terms. The ``finite
 domain'' in the title refers to the fact that the lengths of
 computations and lengths of proofs of instances of the problems that
 we consider are at most exponential, hence there is a \emph{finite}
 bound on them. Perhaps, a more precise term would be ``exponential
 domain''. In previous presentations of this topic, in particular
 in~\cite{godel100,kniha}, we used the term \emph{``feasible incompleteness'',} which
 should be understood as \emph{``being incomplete with respect to feasible
 proofs''.} In~\cite{godel100,kniha} we also stated \emph{the feasible
   incompleteness thesis}, which is an informal statement saying that
 unprovability of a sentence in a weak formal system may be caused by high computational
 complexity of a computational problem naturally associated with the sentence.

\medskip
This paper is partly a survey, but a large part consists of new results, or results that have not been published with full proofs. Specifically, the main  conjectures were already presented in~\cite{kniha}, but connections between them were only sketched there. 
 
Here is how the paper is organized. After two introductory sections, in Section~\ref{sec3}, we recall a conjecture about finite consistencies and introduce a new conjecture about finite reflection principles. In Section~\ref{sec4} we present another important conjecture about total polynomial search problems. We discuss equivalent and stronger statements based on propositional proof systems and disjoint {\bf NP} and {\bf coNP} pairs of sets in Section~\ref{sec5}. We introduce a classification of conjectures in Section~\ref{sec6} and  show that uniform conjectures can be stated as statements about unprovability, which suggests a way towards general conjectures. Section~\ref{sec8} is about the role of reductions in the statements of conjectures. We conclude the paper with some
open problems.

%% we review some conjectures that we presented
%% earlier~\cite{KP,godel100,kniha,herbr} and introduce a new
%% Conjecture~\ref{c3} that naturally fits into our scheme. Some of the
%% conjectures are statements that had been studied before in
%% computational complexity. We compare our conjectures with some other
%% statements that had been studied before in computational complexity.
%% In the second half, we introduce a classification of the the
%% conjectures in order to fully understand what the statements have in
%% common and how they possibly can be generalized. Then we show that
%% most conjectures can be reformulated in terms of unprovability,
%% which may help find a general principle on which they are based. We
%% also briefly look at the role of concepts used in the statements, in
%% particular, the role of reductions. We conclude the paper with some
%% open problems.

\section{Preliminaries}\label{sec2}

\subsection{Theories}

In this paper we will use the word \emph{``theory''} for a set of axioms that is decidable in polynomial time (i.e., for each formula we can decide in polynomial time in the length of the formula whether or not it is an axiom). This implies that given a sentence $\phi$ and a string of symbols $d$, it is possible also to decide in polynomial time if $d$ is a proof of $\phi$.
Furthermore, we will only consider \emph{consistent arithmetical theories} that use a fixed finite set of function and relation symbols representing functions and relation on the natural numbers. We use fragments of arithmetic (as those theories are called), because one can easily refer to standard formalizations of basic syntactical concepts. Being able to formalize syntactical concepts, such as first order formulas and proofs, is the essential property of the theories that we need. 

Furthermore, we need that theories be \emph{sufficiently strong}, because we need the formalizations of basic properties of syntactical concepts and computations to to be provable in the theories we will use. As usual, we will ensure that a theory is sufficiently strong by assuming that it contains a particular fixed basis theory. We will use Buss's theory $S^1_2$ for this purpose. Theory $S^1_2$ is one of the fragments
of Bounded Arithmetic $S_2$ defined by Buss~\cite{buss86} (see also \cite{HP,krajicek95}). Formally, it is not a fragment of $PA$
(Peano Arithmetic), because it is formalized in a slightly richer
language, but it is interpretable in it. 

\bdf 
We denote by $\cal T$ the class of all {consistent} arithmetical theories that extend Buss's theory $S^1_2$ by a set of axioms that is decidable in polynomial time.
\edf

For lack of a good name, we will only use the symbol $\cal T$ to
denote this class of theories.

\subsection{Bounded arithmetic}

We will briefly describe $S^1_2$. It is very convenient to use $S^1_2$, but it should be noted that essentially all results and conjectures do not depend on the particular choice of the base theory. (This also concerns the formalization of the class ${\bf P}$; the particular formalization that we use is also not essential.) So the reader can safely skip this subsection if they are not interested in bounded arithmetic.

$S^1_2$ is a basic fragment of $S_2$ and it has similar relation to $S_2$ as $I\Sigma_1$ (Peano Arithmetic with induction restricted to $\Sigma_1$ formulas) to Peano Arithmetic. In $S_2$ (and so in $S^1_2$) the standard language of arithmetic is enriched by the symbols $\lfloor x/2\rfloor,|x|,x\# y$. The intended interpretation of $\lfloor x/2\rfloor$ is clear; this symbol is used in induction axioms. $|x|$ is the length of the binary representation of $x$ if $x>0$ and $|0|=0$. The interpretation of $x\# y$ is $2^{|x|.|y|}$. Note that we do not have exponentiation in $S_2$, so $\#$ has to be a primitive symbol. Also note that the length of the binary representation of  $2^{|m|.|n|}$ is roughly the product of the lengths of $m$ and $n$. One can easily show that if $t(x_1\dts x_k)$ is a term in the language of $S_2$, then the length of $t(n_1\dts n_k)$ is bounded by $p(n_1\dts n_k)$ for some polynomial, and, vice versa, if $f(x_1\dts x_k)$ is a function that increases the lengths of input numbers at most polynomially, then there exists a term $t$ such that $f(x_1\dts x_k)\leq  t(x_1\dts x_k)$. 

The theory $S_2$ is axiomatized by a finite set of axioms BASIC that fix the intended interpretation of symbols and induction axioms
\bel{PIND}
\alpha(\bar 0)\wedge \forall x(\alpha(\lfloor x/2\rfloor)\to \alpha(x))\to\forall x.\alpha (x),
\ee
for all bounded formulas $\alpha$. %% (One can state these axioms without $\lfloor\dots\rfloor$ by replacing the middle formula with 
%% $\forall x(\alpha(x)\to(\alpha(\bar 2x)\wedge\alpha(\bar 2x+1)))$.
$S^i_2$ is $S_2$ with the axiom schema restricted to $\Sigma^b_i$ formulas; we will define the classes $\Sigma^b_i$ below.

\subsection{Formulas and complexity classes}\label{ssec-2.3}

By a bounded formula, we mean a formula in which quantified variables are bounded by terms in the language of $S_2$. As we noted above, $x\leq t(y_1\dts y_k)$ implies that the \emph{length} of $x$ is polynomially bounded by the lengths of $y_1\dts y_k$. Sometimes we will also need to bound \emph{the number itself} by a polynomial in the lengths of some other numbers. For example, we may need to bound the number of steps of an algorithm that has as an input the binary representation of a number $x$. In such cases we use \emph{sharp bounds} which are bounds of the form $x\leq |s(y_1\dts y_k)|$. Since the outer function symbol in the term is $|\dots|$,  $x$ is polynomially bounded by the \emph{lengths} of $y_1\dts y_k$. \emph{Sharply bounded quantifiers} are bounded quantifiers with sharp bounds.

The hierarchy of bounded formulas $\Sigma_n^b,\Pi^b_n$, $n=1,2,\dots,$ is defined by counting alternations of bounded quantifiers while \emph{ignoring sharply bounded quantifiers}. In particular, prenex formulas that use bounded existential quantifiers and arbitrary sharply bounded quantifiers are $\Sigma^b_1$. So in $\Sigma^b_1$ (and similarly in higher classes) bounded existential quantifiers may alternate with sharply bounded universal quantifiers. This complication can be avoided by slightly extending $S^1_2$ with more function symbols and axioms. If we do this, then we can move all sharply bounded universal quantifiers after the bounded existential ones. The $\Sigma^b_n$ formulas where all sharply bounded quantifiers are after all bounded quantifiers are called {\it strict-}$\Sigma^b_n$, or  $\hat{\Sigma}^b_n$ formulas. $\hat{\Pi}^b_n$ formulas are defined similarly.

In order to simplify formulas, we will sometimes use quantifiers with a superscript $\forall^p,\exists^p$ to indicate that the lengths of the quantified variables are polynomially bounded in the formula that follows. For example, $\forall x\exists^p y.\phi(x,y)$ means that $\phi(x,y)$ is equivalent to the formula $|y|\leq p(|x|)\wedge\phi(x,y)$ for some polynomial $p(x)$ (or, equivalently, to  $y\leq t(x)\wedge\phi(x,y)$ for some $S_2$ term $t$).

The subsets of $\N$ that are in {\bf NP} are precisely those that are definable by $\Sigma^b_1$ formulas. Similarly, other classes from the hierarchy of formulas $\Sigma_n^b,\Pi^b_n$ define corresponding complexity classes $\Sigma_n^p,\Pi^p_n$ from the \emph{Polynomial Hierarchy}. 

For {\bf P}, there is no simple definition of a class of
formulas. Formulas from the class $\Sigma^b_0(=\Pi^b_0)$ have only
sharply bounded quantifiers. These bounds imply that they define sets
and relations computable in polynomial time, but we cannot define all
sets in {\bf P} by such formulas. The standard approach is to extend
the language by function symbols for every polynomial time algorithm
as it is in Cook's theory $PV$~\cite{cook75}.%
\footnote{The relation of $PV$ to $S^1_2$ is similar to
the relation of Primitive Recursive Arithmetic to Peano Arithmetic.} 
This requires also adding infinitely many axioms specifying the intended interpretation of each function symbol.  In this paper we will use a different approach, one that does not need an infinite number of function symbols and axioms. To this end we will use \emph{Buss's Witnessing Theorem.}

\begin{theorem}[\cite{buss86}]\label{t-witnessing}
Let $\phi(x,y)\in\Sigma^b_1$ and suppose that $S^1_2\vdash\forall x\exists y.\phi(x,y)$. Then there exists a polynomial time computable function $f$ such that $\N\models\forall x.\phi(x,f(x))$. Moreover, $f$ is definable by a $\Sigma^b_1$ formula.
\end{theorem}
The definability of $f$ means that there exists a $\Sigma^b_1$ formula $\psi(x,y)$ such that 
$$
S^1_2\vdash\forall x\exists! y.\psi(x,y)\wedge\forall x\forall y(\psi(x,y)\to\phi(x,y)).
$$

A formula $\sigma(x)$
is $\Delta^b_1$ provably in a theory $T$ if $\sigma(x)\in\Sigma^b_1$
and, for some $\pi(x)\in\Pi^b_1$, $T$~proves the
sentence $\forall x.\sigma(x)\equiv\pi(x).$  By Buss's
Witnessing Theorem, the provability of the equivalence in $S^1_2$
ensures that $\sigma(x)$ defines a set in~{\bf P}.  We should stress
that it is essential that the proof is in $S^1_2$. The equivalence
$\N\models\forall x.\sigma(x)\equiv\pi(x)$ in general only ensures that
$\sigma(x)$ defines a set in ${\bf NP\cap coNP}$ which is believed to
be larger than {\bf P}.
%% Thus when we will speak about a \emph{formalization of a
%%   polynomial predicate or relation,} we will mean a formula $\sigma(x)$ that comes from a triple  satisfying the
%% condition above.
%% (Again, this particular representation of predicates in {\bf P} is not essential.) 

Thus we will formalize polynomial decidable sets and relations by formulas that are  $\Delta^b_1$ provably in $S^1_2$. In the rest of this paper \emph{$\Delta^b_1$ will always mean: $\Delta^b_1$ provably in $S^1_2$.} Polynomial time computable functions will be formalized by $\Sigma^b_1$ formulas $\psi(x,y)$ such that $S^1_2\vdash\forall x\exists! y.\psi(x,y)$. One can show that polynomial time computable functions are provably in $S^1_2$ closed under composition and a form of recursion called \emph{limited recursion on notation}~\cite{cook75,buss86,krajicek95}. Because of these properties, one can formalize syntax in a natural way in  $S^1_2$.

%% We will also need to represent binary relations computable in polynomial time. The definition of the corresponding $\Delta^b_1$ formulas is the same, except that one uses two variables instead of one. Polynomial time computable functions will be formalized by $\Delta^b_1$ formulas defining the graphs of these functions.

%% \emph{$\Delta^b_1$ will always mean provably in $S^1_2$.} 

%% ; essentially, any natural syntactical definition would do.

% too many essentials!

\subsection{Proofs}

We will assume that proofs are formalized in a standard Hilbert-style
proof system for first order logic.  We will view the proofs as
strings of formulas such that each formula is either an axiom (logical
or an axiom of the theory in question) or is derived from previous
formulas by an application of a deduction rule. The particular choice
of the system makes little difference, because various calculi for first order logic \emph{polynomially simulate} each other, which means that there are polynomial time computable transformations of proofs from one calculus to the other. However, it is important that the graph of the proof is a general directed acyclic graph, not just a tree, because transforming a general proof into the tree form may increase the length exponentially. 

We could also use sequent calculi, but only the calculi with the cut rule present. Cut-elimination may increase the length more than any elementary function. 

We will assume that formulas and proofs are encoded by binary strings. The \emph{length of a proof} is the length of the string representing the proof. For the \emph{G\"odel number of a formula, or a proof}, we add 1 at the beginning of the string and take the number that it represents in binary notation.

A proof in a theory $T$ will simply be called a \emph{$T$-proof}.

\subsection{Notation}
%{Binary numerals and Turing machines.} 

A \emph{binary numeral} is a suitably
chosen closed term $\bar n$ whose value is $n$ and whose length is
$O(\log n)$. For example, we can represent a number with binary representation $a_1a_2\dots a_k$ by the term
\[
(\dots((a_1\cdot \bar 2)+a_2)\cdot \bar 2\dots\quad \dots+a_{k-1})\bar 2+a_k,
\]
where $\bar 2$ is $SS(0)$.

Our computation model is the standard Turing
machine, where the inputs are words in the alphabet $\Sigma:=\{0,1\}$. When
computing with numbers, we assume the binary representation. We will
use numbers instead of binary strings when we formalize computations.
For $n\in\N$, we denote by $|n|$ the length of the binary
representation of $n$ (the same symbol as is in $S_2$).

We will denote by $Pr_T(x,y)$ a natural formalization of the relation \emph{``$x$ is a $T$-proof of $y$''}. We will assume that basic properties of this relation are provable in $S^1_2$. Furthermore, we will assume that the following fact is true.

\begin{fact}\label{f-1}
If $m$ is a G\"odel number of a $T$-proof of a sentence with a G\"odel number $n$, then there exists an $S^1_2$-proof of $Pr_T(\bar m,\bar n)$ whose length is bounded by a polynomial in $|m|$ and $|n|$.
\end{fact}

The main numerical parameter will be denoted by $n$. When we say \emph{``polynomial length''} without mentioning the argument of the polynomial, we will always mean \emph{``length bounded by $p(n)$ for some polynomial $p$''}.

\section{The basic paradigm -- finite consistency}\label{sec3}

\subsection{Finite consistency}

Let $T\in\cal T$. We will denote by $Con_T(x)$ a formula expressing (in
a natural way) the fact that there is no $T$-proof of contradiction of
length $x$. In particular, we will need $Con_T(\bar n)$, for $n\in\N$. The question mentioned in the introduction is:

\begin{que}
What is the length of the shortest $T$-proof of $Con_T(\bar n)$?
\end{que}

Using the analogy with G\"odel's incompleteness theorem, it is
natural to conjecture that the proof must be long, specifically, not
polynomial in $n$. Friedman also proved a lower bound $n^\epsilon$ for some
$\epsilon>0$.%
\footnote{Note that the length of the sentence  $Con_T(\bar n)$ is $O(\log n)$.}
 This lower bound was improved to $\Omega(n/\log^2 n)$
for a proof system with Rosser's \emph{C-rule} 
\[
\frac{\exists x.\phi(x)}{\phi(c)}
\]
where $c$ is a new constant~\cite{pudlak87}. This rule enables one to
refer to an element satisfying $\phi$ without having to mention
$\phi$. The same asymptotic bound is probably true for some other
systems where this rule can be simulated.  In particular, in natural
deduction systems, we can start with an assumption $\phi(y)$ and argue
about $y$ without having to repeat the assumption in each proof line.

The idea of the proofs of these lower bounds is to adapt the original
proof of G\"odel to the finite setting. Thus instead of the original
diagonal formula, one uses a formula $\delta(\bar n)$ with intended
meaning {\it ``I do not have a $T$-proof of length $\leq n$''}. One
can easily prove that $\delta_T(\bar n)$ is true and any proof of it
must be longer than $n$. Then one proves that $\delta_T(\bar n)$ can
be derived from $Con_T(\bar n)$ by a short proof. This is essentially
the same as in the proof of G\"odel's theorem, except that one has to
prove good upper bounds on the lengths of proofs of certain
true sentences. The shorter proofs one is able to find, the larger the
lower bound is.

In \cite{pudlak86} a polynomial \emph{upper bound} was proved for finitely axiomatized sequential
theories. In \cite{pudlak87} the bound was improved to a linear upper bound for finitely axiomatized sequential theories and proofs using the C-rule.  
Sequential theories are, roughly speaking, theories in which one can code any finite sequence of elements of the universe (see~\cite{HP} for the definition). Already very weak fragments of arithmetic and set theory are known to be sequential. 
This bound is based on partial truth definitions. In the
standard proofs of the consistency of a theory $T$ (without any bound on
the lengths of proofs), one uses a truth definition for all
formulas. Since in proofs of bounded length only formulas of bounded
complexity can occur, it suffices to use a partial truth definition
that defines truth only for sentences of limited complexity. The fact
that partial truth definitions exist is well-known. However, to obtain such bounds one has to carefully estimate the length of the formulas and the lengths
of proofs of particular statements. A polynomial upper bound can also be proved for theories axiomatized by a schema of a particular form, which includes Peano Arithmetic and Zermelo-Fraenkel Set Theory~\cite{pudlak86}.

To sum up the discussion above we state the bounds explicitly, but for the sake of simplicity we will only use theories from $\cal T$.

\bt[\cite{pudlak86,pudlak87}]\label{t-bounds}
For every theory $T\in\cal T$, there exist $\epsilon>0$ such that 
the length of the shortest $T$-proof of $Con_T(\bar n)$ is at least $n^\epsilon$. Moreover, if $T$ is sequential and finitely axiomatized, then there are $T$-proofs of $Con_T(\bar n)$ lengths of polynomial in $n$.
\et

In spite of the polynomial upper bound, we still believe that the incompleteness phenomenon of G\"odel's theorem should manifest itself also in the finite domain---manifest not only by the $n^\epsilon$ lower bound. We conjecture that if $T$ is sufficiently \emph{stronger} than a theory
$S$, then $S$-proofs of $Con_T(\bar n)$ cannot be polynomially bounded. Concerning G\"odel's Theorem, it should be noted at this point that the fact that $T$ does not prove its own consistency is not important---having a proof of consistency in a theory that we do not a priory believe is consistent would be useless. What is important is the consequence of G\"odel's Theorem that there is no theory that could prove the consistency of all other consistent theories. The paradigm for our conjecture is this corollary, not the theorem itself.

Since it is not clear how much stronger $T$ must be, we proposed the following conjecture in~\cite{pudlak86}:

\begin{conjecture}{\sf CON$^N$}\label{CONN}
  For every $S\in\cal T$, there exists $T\in\cal T$ such that the lengths of 
  $S$-proofs of $Con_T(\bar n)$ cannot be bounded by a polynomial in $n$.%
\footnote{The superscript $N$ stands for ``nonuniform'', whose meaning will be explained in Section~\ref{sec6}.}
\end{conjecture}

% Note that we do not have any examples of two theories $S,T\in\cal T$
% such that $S$-proofs of $Con_T(n)$ cannot be polynomially bounded.
% We do not have any lower bounds on how much stronger $T$ must be. So
% the following stronger%

Of course, we would also like to know how much stronger $T$ must be
than $S$ so that there are no polynomial length $S$-proofs of
$Con_T(n)$. It has been conjectured that it suffices that $T$ proves
the consistency of $S$, i.e., the following seems to be true:
\begin{conjecture}{\sf CON$^{N+}$}
 for every $S,T\in\cal T$, if $T$ proves $Con_S$, then
  $S$-proofs of $Con_T(\bar n)$ cannot be bounded by a polynomial in $n$.
\end{conjecture}

The following observation suggests that the assumption about provability of the consistency may be the right choice.

\begin{proposition}
Let $T_0^{Con}:=T$ and $T_{k+1}^{Con}:={T_k^{Con}+Con_{T_k^{Con}}}$ for $k\in\N$. Suppose that for every $T\in\cal T$, $T$ proves  $Con_{T+Con_T}(\bar n)$ by proofs of polynomial length. Then every $T\in\cal T$ proves $Con_{T_k^{Con}}(\bar n)$ by proofs of polynomial length for every $k\in\N$.
\end{proposition}
In plain words, if there are polynomial length $T$-proofs of  $Con_{T+Con_T}(\bar n)$ for every $T\in\cal T$, then there are polynomial length $T$-proofs of bounded consistencies of theories essentially stronger than $T+Con_T$. The proposition is an immediate corollary of the following lemma.

\bl
Suppose that $R$ proves  $Con_s(\bar n)$ by proofs of polynomial length and $S$ proves  $Con_T(\bar n)$ by proofs of polynomial length. Then $R$ proves  $Con_T(\bar n)$ by proofs of polynomial length. (Polynomial bounds are in terms of $n$.)
\el
\prfs
Instead of polynomially bounded proofs of bounded consistency statements, suppose that $R\vdash Con_S$ and $S\vdash Con_T$. Arguing in $R$, suppose that $\neg Con_T$. By the $\Sigma$-completeness of $S$, we have $S\vdash\neg Con_T$, hence $S$ is inconsistent, contrary to the assumption $R\vdash Con_S$.%
\footnote{A high-level proof is: $Con_S$ is the $\Pi_1$-reflection principle for $S$ and $Con_T$ is a $\Pi_1$ sentence. Hence $R\vdash Con_T$.}
To prove the lemma just restate the proof with bounded consistency statements and polynomial bounds on the proofs.
\qed

It is well-known~\cite{ehrenfeucht-mycielski} that if $T$ is stronger (proves more
sentences) than $S$, then some sentences provable in both theories
have much shorter proofs in $T$ compared to the proofs in~$S$. This may suggest that it would
suffice to make $T$ just a little stronger than $S$ in order to ensure
that $S$-proofs of $Con_T(\bar n)$ do not have polynomial length 
proofs. However, recently Pavel Hrube\v{s} proved, using a Rosser-type
selfreferential sentence, that in general this is not the case [personal communication].
His result is even stronger than the mere refutation of such a strengthening of Conjecture~{\sf CON$^{N+}$}.

\begin{theorem}[P. Hrube\v{s}, unpublished]
Let $S\in\cal T$ be a sequential finitely axiomatized theory %, or theory axiomatized by a schema. 
and let $S\sub T\in\cal T$. Then  there exists a true $\Pi_1$ sentence $\phi$  such that $\phi$ is not provable in $T$, yet the lengths of  $S$-proofs of $Con_{S+\phi}(\bar n)$ can be bounded by a polynomial in~$n$.
\end{theorem}
% In particular, if $S=T$, we get $\phi\in\Pi_1$ unprovable in $S$ with polynomially bounded $S$-proofs of $Con_{S+\phi}(\bar n)$.

\bprf
By the fixpoint theorem, one can construct a sentence $\phi$ such that
\bel{fixpoint}
S\vdash
\phi\ \equiv\ \forall x(Pr_T(x,\bar{\phi})\to
  \exists y(|y|\leq 2^{|x|}\wedge Pr_S(y,\neg\bar{\phi}))),
\ee
where $Pr_T(u,v)$ (and $Pr_S(u,v)$) are natural formalizations of the relation $u$ is a $T$-proof (respectively $S$-proof) of $v$. 
This is the well-known Rosser sentence, except that we use the bound $|y|\leq 2^{|x|}$ instead of the usual one $y<x$ and two proof relations instead of one. One can easily prove
\[
T\not\vdash\phi,\ S\not\vdash\neg\phi,\mbox{ and }\N\models\phi,
\]
in the same way as it is done for the standard Rosser sentence~(cf.~\cite{HP}). The idea of the proof of the theorem is to adapt the proof of $S\not\vdash\neg\phi$ so that it gives a polynomial length $S$-proof of $Con_S(\overline{p(n)})\to \neg\exists x(Pr_{S}(x,\neg\bar\phi)\wedge |x|\leq\bar n)$ for some polynomial $p$. Note that the consequent of the implication is essentially $Con_{S+\phi}(\bar n)$. Hence we get the statement of the theorem from the upper bound of  Theorem~\ref{t-bounds}.

In the rest of this proof we will use the symbol 
$\dok$ to denote provability by a proof of length polynomial in $n$.

\begin{lemma}\label{l-3.5}
$S\dok \neg\exists x(Pr_T(x,\bar{\phi})\wedge 2^{|x|}<\bar n)$, or equivalently,
$S\dok \forall x(Pr_T(x,\bar{\phi})\to 2^{|x|}\geq\bar n)$.
\end{lemma}
\prfs
The number of numbers $x$ such that $2^{|x|}<n$ is at most $n-1$. Thus $S$ can verify that each of them is not a $T$-proof of $\phi$ using a proof whose total length is polynomial in $n$. We are using the fact (not provable in $S$) that $T\not\vdash\phi$.
\qed
Now we continue with the proof of the theorem. 
We will abbreviate the formula $\exists y(Pr_S(y,x)\wedge |y|\leq z)$ by $Prl_S(z,x)$. 
%%
%% Lemma~\ref{l-3.5} gives a concrete $S$-proof of $\neg\exists x(Pr_T(x,\overline{\phi})\wedge 2^{|x|}<\bar n)$ of polynomial size; let $K$ be its G\"odel number. $S$ can verify that $K$ is a G\"odel number of this sentence, so we have
%% \[
%% S\dok Pr_S(\bar K,\lceil\forall x(Pr_T(x,\overline{\phi})\to 2^{|x|}\geq\bar n)\rceil)\wedge |\bar K|\leq\bar m_1,
%% \]
%% where $m_1$ is polynomially bounded by $n$. Hence we also have
%% \bel{e-y}
%% S\dok \exists u(Pr_S(u,\lceil\forall x(Pr_T(x,\overline{\phi})\to 2^{|x|}\geq\bar n)\rceil)\wedge |u|\leq\bar m_1).
%% \ee
%%
We have
%% \bel{e-x}
%% S\dok \exists y(Pr_S(y,\overline{\neg\phi})\wedge |y|\leq\bar n)\to
%% \exists z(Pr_S(z,\lceil\exists y(Pr_S(y,\overline{\neg\phi})\wedge |y|\leq\bar n)\rceil)\wedge |z|\leq\bar m_2),
%% \ee
\bel{e-x}
S\dok Prl_S(\bar n,\overline{\neg\phi})\to
Prl_S(\bar m_2,\lceil(Prl_S(\bar n,\overline{\neg\phi}))\rceil),
\ee
where $m_2$ is polynomially bounded by $n$. This follows by the principle: given a proof of length $\leq n$, $S$ can prove that such a proof exists by a proof of length polynomial in $n$ (see Fact~\ref{f-1}).
Observe that according to~(\ref{fixpoint})
%% \[
%% S\dok (\forall x(Pr_T(x,\bar{\phi})\to 2^{|x|}\geq\bar n)\wedge 
%% \exists y(Pr_S(y,\overline{\neg\phi})\wedge |y|\leq\bar n))\to
%% \phi.
%% \]
\[
S\dok (\forall x(Pr_T(x,\bar{\phi})\to 2^{|x|}\geq\bar n)\wedge 
Prl_S(\bar n,\overline{\neg\phi}))\to
\phi.
\]
(This is in fact provable by a logarithmic length proof.) From Lemma~\ref{l-3.5}, we get
%% \[
%% S\dok \exists y(Pr_S(y,\overline{\neg\phi})\wedge |y|\leq\bar n)\to
%% \phi.
%% \]
\[
S\dok Prl_S(\bar n,\overline{\neg\phi})\to
\phi.
\]
Again, by formalizing this proof,
%% \[
%% S\dok \exists z(Pr_S(z,\lceil\exists y(Pr_S(y,\overline{\neg\phi})\wedge |y|\leq\bar n)\rceil)\wedge |z|\leq\bar m_2)\to
%% \exists v(Pr_S(v,\overline{\phi})\wedge |v|\leq \bar m_3),
%% \]
\[
S\dok Prl_S(\bar m_2,\lceil Prl_S(\bar n,\overline{\neg\phi})\rceil)\to
Prl_S(\bar m_3,\overline{\phi}),
\]
for some $m_3$ polynomially bounded by $n$. By (\ref{e-x}), this reduces to 
%% \[
%% S\dok \exists y(Pr_S(y,\overline{\neg\phi})\wedge |y|\leq\bar n)\to
%% \exists v(Pr_S(v,\overline{\phi})\wedge |v|\leq \bar m_3).
%% \]
\[
S\dok Prl_S(\bar n,\overline{\neg\phi})\to
Prl_S(\bar m_3,\overline{\phi}).
\]
Since the antecedent says that $\neg\phi$ is provable by a proof of length $\leq n$, we get
%% \[
%% S\dok \exists y(Pr_S(y,\overline{\neg\phi})\wedge |y|\leq\bar n)\to
%% \neg Con_S(\bar m),
%% \]
\[
S\dok Prl_S(\bar n,\overline{\neg\phi})\to
\neg Con_S(\bar m),
\]
for some $m$ polynomially bounded by $n$. Since $Con_S(\bar m)$ has an $S$-proof of length polynomial in $m$, we get a polynomial length $S$-proof of the negation of the antecedent, hence also of $Con_{S+\phi}(\bar n)$.
\eprf

Finally, we observe that if $T$ is not finitely axiomatized, then it is possible that it does not prove the sentences $Con_T(\bar n)$ by proofs of polynomial (in $n$) lengths. I am indebted to Fedor Pakhomov for suggesting this problem.

\bpr
Suppose {\sf CON}$^N$ is true. Then for every $S\in\cal T$, there exists $S'$, $S\sub S'\in\cal T$ such that the lengths of $S'$-proofs of $Con_{S'}(\bar n)$ cannot be bounded by a polynomial in $n$.
\epr
\bprf
Let $S\in\cal T$ be given. Let $T\in\cal T$ be such that the lengths of 
$S$-proofs of $Con_T(\bar n)$ cannot be bounded by a polynomial in $n$. Such a $T$ exists if we assume {\sf CON}$^N$. Define
\[
S' := S \cup \{ \neg Pr_T(\bar N,\bot)\ |\ N\in\N \}.
\]
Suppose that $S'$ proves $Con_{S'}(\bar n)$ by a proof of length $p(n)$ for some polynomial $p$. Such a proof can only use a polynomial number of the axioms of the form
$\neg Pr_T(\bar N,\bot)$ with $|N|\leq p(n)$ (while there are exponentially many such axioms). But $S$ can prove each of these axioms by a proof of polynomial size. Hence
\[
S\vdash Con_{S'}(\bar n)
\]
by a polynomial size proof. Furthermore,  for a suitable polynomial $q$,
\[
S\vdash Con_{S'}(\overline{q(n)}) \to Con_T(\bar n)
\]
by a polynomial (in fact, even polylogarithmic) length proof. To prove this claim, one only needs to formalize the following argument in $S$:
\begin{quote}
``Suppose there exists a proof $P$ of contradiction in $T$ of length $\leq n$. Then it is possible to prove in $S$ that $Pr_T(\bar P,\bot)$ holds true by a proof of length polynomial in $|P|$, i.e., also polynomial in $n$. Hence we would get a proof of contradiction in $S'$ of length $q(n)$.''
\end{quote}
Thus we would get polynomial length $S$-proofs of $Con_T(\bar n)$ contrary to our assumption.
\eprf

\subsection{A finite reflection principle}

Recall that the sentences expressing consistency of a theory $T$ are
special cases of \emph{reflection
  principles} (see~\cite{smorynski77}). There are many versions of
reflection principles. Here we will focus on the uniform
$\Sigma_1$-reflection principles.

%% Let a universal $\Sigma_1$ formula $\sigma_1(x)$ be fixed. 
%% (This means that for every $\Sigma_1$ sentence $\phi$ there exists $n\in\N$ such that $\phi\equiv\sigma_1(\bar{n})$ is provable in a fixed base theory.)
The
\emph{uniform $\Sigma_1$-reflection principle for $T$,} $\Sigma_1 RFN_{T}$, is the
following schema for all $\Sigma_1$ sentences $\sigma(x)$ with one free variable $x$
\[
\forall x\forall u(Pr_T(u,\lceil\sigma(\bar{x})\rceil)\to \sigma(x)),
\]
where $\lceil\sigma(\bar x)\rceil$ denotes the function that assigns the G\"odel number of the formula  $\sigma(\bar x)$ to a given $x$, and
$Pr_T(u,v)$ says that $u$ is a proof of $v$ in
$T$. The principle is true if $T$ is $\Sigma_1$-sound, i.e., $T$ does not prove a false $\Sigma_1$ sentence. This schema can be axiomatized by a \emph{single sentence} using a partial truth definition for $\Sigma_1$ formulas (a universal $\Sigma_1$ formula). Therefore the principle is called {\it ``uniform''} and abbreviated with capital letters.
%% $\Sigma_1 RFN_{T}$ is a
%% $\Pi_2$ sentence and every $\Pi_2$ sentence provable in $T$ can be
%% derived from an instance of the principle in a weak base theory
%% (see~\cite{smorynski77}).

In order to get a meaningful finite version of  $\Sigma_1 RFN_{T}$ we
have to make a couple of modifications. We start by defining a finite
$\Sigma_1^b$ reflection principle for one formula.

\begin{definition}
Let $T$ be a theory, let $\alpha(x)$ be a $\Sigma^b_1$ formula and
let $n\in\N$. Then
{$\Sigma_1^bRfn_T^\alpha(\bar n)$} 
will denote the sentence:
\[
\forall u,x, |u|\leq \bar{n},|x|\leq \bar{n} 
\left(Pr_T(u,\lceil\alpha(\bar{x})\rceil)
\to \alpha(x)\right).
\]
\end{definition}

Having defined the reflection principle for one formula, we can study
the schema, i.e., the set of sentences 
$\Sigma_1^bRfn_T^\alpha(\bar n)$ for all $\Sigma^b_1$ formulas, 
but it is more interesting to have 
a single sentence for every $n$ from which all instances are derivable by short proofs. 
%% In the finite case that we are
%% looking at this means that for every $n$ we have a single sentence
%% that replaces the special cases (in the sense of Lemma~\ref{l-3.2}
%% below).
To this end we need a universal $\Sigma^b_1$ formula. One can
construct a $\Sigma^b_1$ formula $\mu_1$ such that for every $\Sigma^b_1$ formula
$\alpha(x)$ there exist a natural number $e$ and a polynomial $p$
such that 
\bel{e-univ} 
|z|\geq p(|x|)\to(\alpha(x)\ \equiv\ \mu_1(\bar e,x,z)) 
\ee 
is provable in $S^1_2$ (see~\cite{HP}, page 336). The sentences that we are going to define are essentially $\Sigma_1^bRfn_T^{\mu_1}(\bar n)$.

\begin{definition}
The \emph{finite uniform $\Sigma^b_1$ principle} is the
sequence of sentences $\Sigma_1^bRFN_T(\bar n)$, $n\in\N$, defined by
\[
\forall e,u,x,z(|e|,|u|,|x|,|z|\leq \bar{n}\wedge 
Pr_T(u,\lceil\mu_1(\bar{e},\bar{x},\bar{z})\rceil) 
\to \mu_1(e,x,z)).
\]
\end{definition}

%% The proof that  $\Sigma_1^bRFN_T(\bar n)$ is true for all $n$ and $T\in\cal T$ is the same as for $\Sigma_1^bRfn_T^\alpha(\bar n)$, the only difference being that $\mu_1$ has more parameters than $\alpha$.

\bl\label{l-3.2}
For every $\Sigma^b_1$ formula $\alpha(x)$, there exist polynomials $q$
and $r$
such that  $S^1_2$-proofs of the sentences
\[
\Sigma_1^bRFN_T(\overline{q(n)})\to\Sigma_1^bRfn_T^\alpha(\bar n)
\]
can be constructed in time $r(|n|)$.  
\el 
\bprf 
Let $e\in\N$ and $p$ be such that (\ref{e-univ}) is provable in
$S^1_2$. Let $n\in\N$ be such that $n\geq|e|$ and let $m=p(n)$. The
following argument can be done in $S^1_2$.
\begin{quote}
Suppose $|u|,|x|\leq n$ and $Pr_T(u,\lceil\alpha(\bar{x})\rceil)$. Then we also have
$|u|,|x|\leq m$ and, since (\ref{e-univ}) is provable in $T$, we have 
$Pr_T(u',\lceil\mu_1(\bar{e},\bar{x},\overline{2^m})\rceil)$ for
some $u'$. The proof $u'$ is constructed from $u$ using the proof of
(\ref{e-univ}) in $T$, which adds only a constant to the length and a small
part in which this sentence is instantiated for the numerals $\bar{x}$
and $\overline{2^m}$. This makes the proof $u'$ at most polynomially
longer than $m$. Let $m'$ be this polynomial bound. Applying 
$\Sigma_1^bRFN_T(\overline{m'})$, we get 
$\mu_1(\bar{e},\bar{x},\overline{2^m})$. Then
using (\ref{e-univ}) in $S^1_2$, we finally get 
$\alpha(\bar{x})$.
\end{quote}
Now we only need to observe that the above $S^1_2$ proof was
explicitly constructed and the number of steps and the length of the
formulas involved are of length polynomial in $|n|$.
\eprf

% The following is a basic fact about this principle.

\begin{corollary}\label{l-rfn}
Let $S,T\in\cal T$. Suppose that 
\ben
\item $T\vdash \forall x.\phi(x)$, where $\phi\in\Sigma^b_1$, and 
\item $S$-proofs of the sentences $\Sigma_1^bRFN_T(\bar n)$ can be
constructed in polynomial time in $n$. 
\een
Then $S$-proofs of the sentences 
$\phi(\bar m)$ 
can be constructed in time $r(|m|)$ for some polynomial~$r$.
\end{corollary}
% Then the following holds true.
% \ben
% \item If the sentences $\Sigma_1^bRFN_T(\bar n)$ have $S$-proofs of
%   polynomial (in $n$) length, then the sentences $\exists y.\phi(\bar
%   m,y)$ have $S$-proofs of length $q(|m|)$ for some polynomial $q$.
\begin{proof}
  Since $\forall x.\phi(x)$ is provable
  in $T$, the sentences $\phi(\bar m)$ have
  $T$-proofs of length bounded by $q(|m|)$ for some polynomial
  $q$. This is provable in $S^1_2$, so also in $S$. According to the
  assumption about $S$ and by Lemma~\ref{l-3.2}, one can construct in
  polynomial time proofs of $\Sigma^b_1Rfn^\phi_T(\overline{q(|m|)}$ in
  polynomial time in $|m|$. Thus we get $S$-proofs of $\phi(\bar m)$ in polynomial time.
\end{proof}

% DOES THIS PRINCIPLE HAVE SHORT PROOFS IN $T$? I THINK IT DOES, USING
% THE SAME PROOF AS FOR $Con_T(\bar n)$.

Using $\Sigma_1^bRFN_T(\bar n)$, we can state a conjecture similar to our
conjecture about 
$Con_T(\bar n)$. 

\begin{conjecture}{\sf RFN$^N_1$}\label{RFNN}
  For every $S\in\cal T$, there exists $T\in\cal T$ such that the
  lengths of 
  $S$-proofs of $\Sigma_1^bRFN_T(\bar n)$ cannot be bounded by $p(n)$ for any polynomial $p$.
\end{conjecture}

When $\alpha$ is $0=1$, then $\Sigma_1^bRfn_T^\alpha(\bar n)$ is equivalent
 to $Con_T(\bar n)$ and this equivalence has an $S^1_2$-proof of length polynomial in $|n|$. 
Thus, by Lemma~\ref{l-3.2}, there exists a polynomial $q$ such that $\Sigma_1^bRFN_T(\overline{q(n)})\to Con_T(\bar n)$ has an $S^1_2$-proof of length polynomial in $|n|$. 
%Since $\Sigma_1^bRFN_T(\overline{p(n)})$ implies $Con_T(\bar n)$, 
Consequently, Conjecture~{\sf CON$^N$} implies Conjecture~{\sf RFN$^N_1$}. 

We will prove that Conjecture~{\sf RFN$^N_1$} implies {\bf NP$\neq$\bf coNP}.

\begin{proposition}
If {\bf NP$=$\bf coNP}, then there exists  $S\in\cal T$ such that for
all $T\in\cal T$, the lengths of $S$-proofs of $\Sigma_1^bRFN_T(\bar n)$ can be bounded by $p(n)$ for some polynomial $p$.
\end{proposition}
\begin{proof}
  The basic idea is to take some base theory and add all
  sentences of the form $\Sigma_1^bRFN_T(\bar n)$ that are true as axioms,
  disregarding whether or not $T$ is consistent. Assuming {\bf NP$=$\bf coNP}, it is possible to test these sentences in nondeterministic polynomial time for each $T$. Since the polynomial bound is different for different theories $T$ we have to apply padding.

Here is a sketch of a proof in more detail. Assume {\bf NP$=$\bf coNP}. Instead of just the sentences expressing $\Sigma_1^bRFN_T(x)$, we will consider all formulas with one free variable $x$ with bounded quantifiers in which the bounds are either of the form $|y|\leq p(x)$, where $y$ the quantified variable and $p$ is a term expressing a polynomial, or of the usual form $y\leq t(x)$, where $t$ is a term of $S_2$. Let $\Gamma$ denote the class of such formulas. Our assumption implies that every formula $\phi(x)\in\Gamma$ is equivalent to a formula from this class in which all quantifiers are existential. Consequently, for every $\phi(x)\in\Gamma$, there exists a nondeterministic Turing machine $M$ that accepts only true sentences of the form $\phi(\bar n)$ and runs in time polynomial in $n$. With some additional (and tedious) work, one can show:

\bcl
There exists a nondeterministic Turing machine $M$ that accepts true sentences of the form $\phi(\bar n)$, for $\phi(x)\in\Gamma$ and $n\in\N$, such that for every  $\phi(x)\in\Gamma$, there exists a polynomial $p_\phi$ such that on sentences $\phi(\bar n)$ the machine always stops after  $p_\phi(n)$ steps.
\ecl

For every  $\phi(\bar n)\in\Gamma$ and $n,m\in\N$, let a padded version of  $\phi(\bar n)$ be
\[
\phi(\bar n)_m:= \phi(\bar n)\vee(\underbrace{0=1\wedge\dots\wedge 0=1}_m).
\]
Clearly, $\phi(\bar n)$ is derivable from $\phi(\bar n)_m$ in predicate calculus using a proof whose length is polynomial (in fact, linear) in the length of the sentence $\phi(\bar n)_m$.
 Using a simple modification of the machine $M$ from the previous claim, one can prove:

\bcl
There exists a nondeterministic \emph{polynomial time} Turing machine $M'$ that accepts only some true sentences of the form $\phi(\bar n)_m$, for $\phi(x)\in\Gamma$, $n,m\in\N$ and such that for every  $\phi(x)\in\Gamma$, there exists a polynomial $p_\phi$ such that for all $n\in\N$, $M'$ accepts $\phi(\bar n)_m$ for some $m\leq p_\phi(n)$.
\ecl

Let $S'$ be the theory axiomatized by the sentences accepted by $M'$. By construction, every true sentence of the form $\phi(\bar n)$, $\phi(x)\in\Gamma$ has an $S'$-proof of length polynomial in $n$; in particular, this holds true for true sentences of the form  $\Sigma_1^bRFN_T(\bar n)$.
The only issue is that the set of axioms is in {\bf NP} and, maybe, not in {\bf P}. This can also be fixed by padding: we can encode accepting computations into padding. E.g., by using $0=1$ and $1=0$, we can encode an arbitrary bit string and let a padded formula be accepted iff the padding encodes an accepting computation of $M'$ on input $\phi(\bar n)_m$.
\end{proof}

The reason for introducing the conjecture about  $\Sigma_1^bRFN$  is that it
enables us to connect diverging branches of so far postulated
conjectures, as we will see shortly. One can certainly study similar statements based on stronger reflection principles for classes of formulas  $\Sigma^b_2,\Sigma^b_3,\dots$. The strength of these principles decreases with increasing indexes, so they are not interesting if we are looking for stronger conjectures. However, the study of these principles may reveal further interesting connections.

%% In fact, one can show that it suffices to assume
%% {\bf NE}={\bf coNE} in the proposition. (Recall that {\bf E} is the
%% class of languages accepted in time $2^{O(n)}$; similarly are defined the
%% nondeterministic and co-nondeterministic versions of this class.)

%%  Clearly, postulating such conjectures for
%% reflection principles is not a way to find strongest possible
%% conjectures. The reason for introducing this conjecture is that it
%% enables us to connect diverging branches of so far postulated
%% conjectures, as we will see shortly.

\subsection{What is the finite G\"odel theorem?}

We finish this section with a remark concerning the question what
should be called the finite G\"odel theorem. If Conjecture~{\sf CON$^N$}
were proven true, we would certainly advocate calling it the finite
G\"odel theorem. However, one can also argue that the connection is
different. Note that if $T$ proves $Con_S$, then $T$-proofs of
$Con_S(\bar n)$ are very short; they are of logarithmic length in $n$, because
the length of $Con_S(\bar n)$ is logarithmic (recall that we are using
binary numerals) and this sentence follows from $Con_S$ by
substitution (if we formalize $Con_S$ as $\forall x.Con_S(x)$). Using
this fact, we can derive G\"odel's theorem from Friedman's lower bound
$n^\epsilon$  on the lengths of $T$-proofs of $Con_T(\bar n)$. So
Friedman's lower bound can also be viewed as the finite G\"odel
theorem.

Proving G\"odel's theorem in this roundabout way is certainly not
natural, but in some cases it may be useful. Using estimates on finite
consistency statements, we proved~\cite{pudlak90} that $S_2$ does not prove
bounded consistency of the apparently weaker theory $S^1_2$, which ruled
out an approach to the separation problem of these two theories. (Bounded
consistency means that we only consider proofs in which all formulas
are bounded.)

\section{Fast growing functions and hard search problems}\label{sec4}

An important property of first-order theories studied in classical
proof theory is their strength measured by the set of arithmetical
sentences provable in them. Among the arithmetical sentences the most
important role is played by $\Pi_1$ and $\Pi_2$ sentences. 
%% Mention ordinal analysis?
%%  Practically
%% for all naturally defined theories, it holds true that if $T$ proves
%% more arithmetical sentences than $S$, then it also proves more $\Pi_2$
%% sentences.
%  In fact, typically $T$ proves some sentences that are not
% derivable from $S+\Pi_1(T)$, where $\Pi_1(T)$ denotes all $\Pi_1$
% sentences provable in $T$.
 A proper $\Pi_2$ sentence, a
sentence that is not equivalent to a $\Pi_1$ sentence, expresses the
fact that some function is total. Specifically, $\forall x\exists
y.\phi(x,y)$, where $\phi$ is a bounded formula, can be interpreted as saying  that there
exists a computable function such that $\forall x.\phi(x,f(x))$. If we
cannot write it equivalently using a formula $\forall x.\psi(x,y)$,
where in $\psi$ all quantifiers are bounded, then $f$ has to grow
faster than all functions defined by the terms of the
theory. Moreover, for pairs of natural theories $S$ and $T$ with $T$
essentially stronger than $S$, there are provably total
computable functions in $T$ that cannot be bounded by computable
functions provably total in $S$. One can say that {\it ``$T$ proves
  the existence of larger numbers than $S$''.} This intuition can be
made more precise using cuts of nonstandard models of arithmetic in
which the arithmetical theories of $S$ and $T$ are
satisfied: in general, $T$ requires longer cuts than $S$.%
\footnote{Consider cuts that contain a fixed nonstandard element.}

\medskip {\bf Remark.} {\small It is important to realize what ``provably
total'' means. For a given theory and a computable function $f$, we can
always find a $\Sigma_1$ definition for which the totality of $f$ is
not provable (e.g., given a defining formula $\phi(x,y)$, we can
extend it by adding the consistency of $T$, i.e.,  $\phi(x,y)\wedge
Con_T(x)$). So when we say that $f$ is provably total, we mean
that $f$ is provably total \emph{for some $\Sigma_1$ definition of $f$}.}

\subsection{Total polynomial search problems}\label{subsec4.1}

We are interested in the exponential domain, which means that we only
consider functions $f$ such that the length of $f(x)$ is bounded by
$p(|x|)$ for some polynomial $p$, so it does not make sense to compare the
growth rate of the functions. Instead, we study the complexity of these
functions. The class of sentences corresponding to $\Pi_2$ are
$\forall\hat{\Sigma}^b_1$ sentences---the sentences starting with unbounded
universal quantifier followed by a $\hat{\Sigma}^b_1$ sentence. Essentially, this class consists of sentences of the form
\bel{e-total}
\forall x\exists y(|y|\leq p(|x|)\wedge\phi(x,y)),
\ee
where $\phi$ is a formalization of a polynomial time relation (i.e., $\phi\in\Delta^b_1$) and $p$ is
some polynomial. There is a computational task naturally associated
with such sentences. Since this is important, we define it formally.

\begin{definition}
A \emph{total polynomial search problem} is given by a pair $(p,R)$,
where $p$ is a polynomial and $R$ is a binary relation such that
\ben
\item $R$ is decidable in polynomial time,
\item $\N\models \forall x\exists y(|y|\leq p(|x|)\wedge R(x,y))$.
\een
The computational task is, for a given $x$, find $y$ such that $|y|\leq
p(|x|)\wedge R(x,y)$.
\end{definition}
The class of all total polynomial search problems will be denoted by
{\bf TFNP}.%
\footnote{The abbreviation {\bf TFNP} is standard, but is
  rather misleading; the class is not a class of functions and it is
  not defined using {\bf NP} relations. Therefore we used {\bf TPS}
  in~\cite{kniha}.} 
 Here are two examples of {\bf TFNP} problems.

\medskip{\bf Example 1.} This example is based on the Pigeon-Hole Principle, which says that there is no one-to-one mapping from an $N+1$-element set to an $N$-element set. The computational task associated with this principle is: given a mapping from an $N+1$-element set to an $N$-element set, find a ``collision'', which is a pair $x\neq x'$ such that $f(x)=f(x')$. This problem is algorithmically trivial if the mapping is given as a list of pairs $(x,f(x))$. In this case $N$ is less than the input length. However, if the problem is presented so that $N$ is exponential in the input length, no polynomial time algorithm is known. Such a representation can be defined using Boolean circuits, or polynomial time algorithms that compute the function $f$. In fact, researchers in cryptography believe that the problem is hard even if the mapping is from $[N]$ to $[M]$ for $M$ much smaller than $N$. These \emph{hash functions} are used in various protocols.

A {\bf TFNP} problem based on the Pigeon-Hole Principle can formally be defined as follows. Take a polynomial time computable function $f(r,x)$; think of $f$ as a set of polynomial time computable functions of one variable $x$ parametrized by $r$. Define a binary relation computable in polynomial time by
\[
R(r,u):\equiv\ (u\leq r\wedge f(r,u)\geq r)\ \vee\ \exists x,x'\leq r (u=(x,x')\wedge f(r,x)=f(r,x')).
\]
In this formula, $u$ is a witness of the fact that $f$ does not map $\{0\dts r\}$ into $\{0\dts r-1\}$ or a witness of a collision.
A polynomial bound on $|u|$ is determined by  a polynomial bound on the lengths pairs of elements less than $r$.

\medskip{\bf Example 2.} Our second example is based on the problem of factoring integers. Again the problem is nontrivial only if the number to be factored is presented in binary (decimal etc.) notation, in which case it is exponential in the input length. Since the search problem must have a solution for every number $N$, we have to distinguish the cases when $N$ is prime and when it is composite. It is well-known that this is decidable in polynomial time. Formally, we define a binary relation computable in polynomial time by
\[
Q(N,M):\equiv\ N\mbox{ is prime }\vee\ (1<M<N\wedge M\mbox{ divides }N).
\]
The bound on $M$ is simply $|M|\leq |N|$. A solution is any number $M$, $|M|\leq|N|$, if $N$ is prime, or a proper factor if $N$ is composite.

\medskip
Having the concept of a total polynomial search problem, we can now replace the \emph{growth rate} of
functions by the \emph{computational complexity} of finding solutions. Not surprisingly, the situation
is much less clear than in the classical setting. Firstly, we can only
hypothesize about the computational complexity of specific search
problems. But this is what we expected and are ready to face.
Secondly, we do not have a quantitative measure of complexity that we
could apply to this kind of computational problems. We can distinguish
problems for which the task is solvable in polynomial time from those
for which it isn't, but some evidence suggests that there are also
distinct classes of problems that are not solvable in polynomial time
and have different complexity. To compare the complexity of different
problems, we use reductions. Polynomial reductions are known for decision problems and used, in particular, in the theory of {\bf NP} completeness. For
{\bf TFNP} there is also a natural concept of polynomial
reduction. (Note that {\bf TFNP} is not a class of decision problems, so we do need
a different concept.)

\begin{definition}[\cite{jpy}]\label{d-5}
Let $R$ and $S$ be total polynomial search problems. We say
that $R$ is \emph{polynomially reducible} to $S$ if $R$ can be solved in
polynomial time using an oracle that gives solutions to $S$. We say
that $R$ and $S$ are \emph{polynomially equivalent} if there are polynomial
reductions in both directions. We say that $R$ is \emph{many-one polynomially
  reducible} to $S$, if it is polynomially reducible using one query
to the oracle for~$S$. 
\end{definition}
Many-one polynomial reducibility can be equivalently defined by the
condition: {\it there are functions $f$ and $g$ computable in polynomial
time such that for all $x$ and $z$,}
\[
S(f(x),z)\ \Rightarrow\ R(x,g(x,z)),
\]
were we are assuming that polynomial bounds on the lengths of numbers
involved are implicit in the relations $R$ and $S$.

Reductions enable us to study the structure of {\bf TFNP} and define subclasses. We are interested in classes that are closed under polynomial reductions. One important class is {\bf PHP}, the class of all {\bf TFNP} problems reducible to an instance of the Pigeon-Hole Problem as described in Example~1 above. Several other classes were defined already in the seminal paper~\cite{jpy}. They enable one to show that a problem is probably not solvable in polynomial time. Specifically, if one proves that a problem is complete in one of the well-known classes, it implies that the problem is not solvable in polynomial time unless the class collapses to the bottom class consisting of all problems solvable in polynomial time. 

From the point of view of computational complexity, it is natural to
identify polynomially equivalent problems. However, we should bear in
mind that from the point of view of a particular theory, two
definition of the same problem may behave differently, as we noted
above. We will consider definitions of {\bf TFNP} by $\Delta^b_1$
formulas and for a given theory we will take ``the best possible
definition''. Formally, this is defined as follows.

\begin{definition}
\ben
% \item A \emph{proper definition of a polynomial search problem
%     $(p,R)$} is a definition where $p$ is defined by an arithmetical
%   term and $R$ is defined by a pair of formulas in
% $\Sigma^b_1$ and $\Pi^b_1$ provably equivalent in $S^1_2$.
\item A $\Delta^b_1$ \emph{definition of a {\bf TFNP} problem} $(p,R)$ is a
  pair $(q,\phi)$ where $q$ is a polynomial and $\phi$ is a 
  $\Delta^b_1$ formula such that 
\[
\N\models \forall x,y((|y|\leq p(|x|)\wedge R(x,y))\equiv
(|y|\leq q(|x|)\wedge \phi(x,y))).
\]
\item We say that $(p,P)\in\ ${\bf TFNP} is \emph{provably total in a theory
  $T$}, if for some $\Delta^b_1$ definition $(q,\phi)$ of $(p,P)$, $T$ proves
that 
\[
\forall x\exists y(|y|\leq q(|x|)\wedge\phi(x,y)).
\]
\item The set of all $(p,P)\in\ ${\bf TFNP} provably total in $T$ will be denoted by {\bf TFNP}$(T)$.
The set of all $P\in\ ${\bf TFNP} polynomially reducible to some
$Q\in${\bf TFNP}$(T)$ will be denoted by ${\bf TFNP^*}(T)$.
\een
\end{definition}
Note that according to our definition of the class $\Delta^b_1$ (in  Subsection~\ref{ssec-2.3}), the
formula $\phi$ must be a $\Sigma^b_1$ formula equivalent to a
$\Pi^b_1$ formula \emph{provably in $S^1_2$} (to ensure that it defines a set
in {\bf P} it does not suffice to have a proof in $T$). On the other hand, we do \emph{not} require that a problem $P$ in ${\bf TFNP^*}(T)$ is \emph{provably} reducible to  some  $Q\in${\bf TFNP}$(T)$. 
The difference
between {\bf TFNP}$(T)$ and ${\bf TFNP^*}(T)$ is small; in fact, if we
defined {\bf TFNP} using {\bf NP} relations (see $\overline{\bf TFNP}$ below), these classes would be the same.
% $R(x,y)\vee\exists z(S(f(x),z)\wedge y=g(x,z))$

To characterize low complexity theorems of fragments of arithmetic is
an important problem studied in proof complexity. In particular, we
are interested in sentences that are universal closures of
$\Sigma^b_1$ formulas. Naturally, we want to identify sentences that express the same fact. The best way to do that is to focus on provably total polynomial search problems. Provably total
polynomial search problems of all fragments of bounded arithmetic $S^i_2$,
$i=1,2,\dots,$  have been characterized using combinatorial
principles~\cite{skelley-thapen,BB10,pudlak-thapen}. ($S^i_2$ is $S_2$ with the induction schema~(\ref{PIND}) restricted to $\Sigma^b_i$ formulas.)
For $S^1_2$ they are all {\bf TFNP} problems that are solvable in polynomial time (the lowest class in  {\bf TFNP}). The class of provably total problems of $S^2_2$ turned out to be surprisingly the class \emph{Polynomial Local Search}, a class that had been introduced in~\cite{jpy}. 

Here is another important conjecture.

\begin{conjecture}{\sf TFNP}\label{TFNP}
For every theory $T\in\cal T$ there exists a {\bf TFNP} problem $P$
that is not polynomially reducible to any {\bf TFNP} problem provably total
in $T$. Stated in symbols  ${\bf TFNP^*}(T)\neq{\bf TFNP}$.%
\footnote{We distinguish the complexity class {\bf TFNP} and the conjecture about it {\sf TFNP} by different fonts.}
\end{conjecture}

The weaker statement ${\bf TFNP}(T)\neq{\bf TFNP}$, in plain words, says that, for every theory $T\in\cal T$, there exists a total polynomial search problem $(p,R)$ such that $T$ cannot prove that the problem is total for any proper definition (definition by a $\Delta^b_1$ formula) of $(p,R)$. This means that the unprovability in $T$ is not caused by a particular way we define the problem, but by a semantic property of it that we imagine as high computational complexity. We state the conjecture in the stronger form, ${\bf TFNP^*}(T)\neq{\bf TFNP}$, because ${\bf TFNP}(T)$ may not be closed under polynomial time reductions.

Let us compare this conjecture with the corresponding statement about
fast growing recursive functions. One can easily prove by
diagonalization that for every $T\in\cal T$, there exists a computable
function $f$ which grows faster than any computable function provably
total in $T$. This means that for any computable function $g$ provably
total in $T$, there exists an $n_0$ such that $f(n)>g(n)$ for all
$n\geq n_0$. Thus for any formalization of $f$ by a $\Sigma_1$ formula
$T$ cannot prove that $f$ is total.  In the above conjecture, the
condition that $f$ cannot be bounded by provably total functions is
replaced by the condition that a {\bf TFNP} problem is not
polynomially reducible to {\bf TFNP} problems that are provably total in $T$.

%  So, again, provably total means, provably total for some
% definition by a $\Delta^b_1$ formula (or a formula from a natural
% class defining polynomial time relations). 

All conjectures in this area can be stated in purely complexity theoretical
terms. The above conjecture has an especially simple equivalent form, which
we state now.

\begin{conjecture}{equivalent to {\sf TFNP}}
There is no complete problem in {\bf TFNP}, i.e., there exist no
{\bf TFNP} problem to which all {\bf TFNP}
problems can be reduced.
\end{conjecture}

The proof of the equivalence of the versions is easy. To prove that the first version implies the second, suppose the second is false. Let $P$ be a complete
problem in {\bf TFNP}. Then take a fragment of arithmetic and add the
axiom that (a formalization of)  $P$ is total. 

The converse
implication follows immediately from the following fact.

\begin{lemma}\label{l-4.1}
For every $T\in\cal T$, there exists a {\bf TFNP} problem $(p,P)$ such that all {\bf
  TFNP} problems provably total in $T$ are many-one polynomially reducible to $(p,P)$.
\end{lemma}
% See \cite{logical} for more details.
\begin{proof}
The proof is based on the fact that one can effectively enumerate all
problems in ${\bf TFNP^*}(T)$. (Such proofs are routine and we
include a proof here only because it demonstrates a method that can be
applied in other similar situations, in particular, we will use it in
Proposition~\ref{pr-6.2}).
The basic idea is to connect all provably total problems into one. We
can recognize a definition of a provably total problem by finding a
proof of the totality for this definition. A minor complication is
that different provably total problems may require different
polynomials as bounds on the witnesses and bounds in the $\Delta^b_1$
formulas defining them. This can easily be solved by suitable padding.

Now we present the argument in more detail. Recall that from the point of view of provability in a theory, it does not matter if we use $\Delta^b_1$ formulas or, more generally, $\Sigma^b_1$ in the definition of the problems. So, for the sake of simplicity, we will enumerate $\Sigma^b_1$ formulas. 

Given a $\Sigma^b_1$
formula $\psi(x)$, we say that $r(n)$ is a
syntactic nondeterministic time bound for $\psi$ if the bounds at quantifiers in the formula ensure that $\psi(x)$ is decidable by a nondeterministic Turing machine in
time $r(n)$ where $n$ is the length of $x$. Since $\psi$ is a
$\Sigma^b_1$ formula, there always exists a polynomial $r$ that is such a bound
for $\psi$.

Let $T\in\cal T$ be given. We define a binary relation $R(u,v)$ by
the following condition:
\bi
\item if $u=(x',\phi,q,d,a)$ is a quintuple such that $\phi$ is a $\Sigma^b_1$ formula, $q$
  is a polynomial, $d$ is
  $|T|$-proof of $\forall x\exists y(|y|\leq q(|x|)\wedge\phi(x,y))$
  and $|a|=r(|x'|)$, where $r$ is a syntactic nondeterministic time bound for
  $\exists y(|y|\leq q(|x|)\wedge\phi(x,y))$, then $\phi(x',v)$.
\ei

Note that $\Phi:=\forall x\exists y(|y|\leq q(|x|)\wedge\phi(x,y))$ is a $\Pi_1$ sentence and $T$ is consistent and $\Sigma_1$-complete (since it contains $S^1_2$). Hence if $T$ proves $\Phi$, then $\Phi$ is true in $\N$.

The relation $R$ is computable in nondeterministic polynomial time, because the
condition on  $(x',\phi,q,d,a)$ is a simple syntactical condition and
if the condition is satisfied,  $\phi(x',v)$ can be computed in
nondeterministic polynomial time bounded by $|a|$. Further, for every
$u$ there exists some $v$, $|v|\leq|u|$, such that $R(u,v)$ holds true,
because if the condition on  $(x',\phi,q,d,a)$ is satisfied, then for
every $x'$ there exists $v$, $|v|\leq|a|$, that satisfies
$\phi(x',v)$, and if the condition is not satisfied, then one can take
$v=0$.

The fact that we only know that $R$ is computable in
\emph{nondeterministic} polynomial time is not a problem. Clearly,
there exists a ternary relation $P'$ computable in polynomial time and
a polynomial $p'$
such that 
\[
R(u,v)\equiv \exists w(|w|\leq p(|u|,|v|)\wedge P'(u,v,w)).
\]
So we define 
\[
P(u,y):\equiv \exists v,w(y=(v,w)\wedge P'(u,v,w))
\quad\mbox{ and }\quad
p(n)=p'(n,n).
\]

Let a  {\bf TFNP} problem $(q,Q)$ be given and suppose that it
is provably total in $T$. We have a 
$\Sigma^b_1$ formula $\phi$ and a polynomial $q$ that defines the
problem and a $T$-proof of totality $d$ for this
representation. Also we have a nondeterministic polynomial time bound $r$
for $\exists y(|y|\leq q(|x|)\wedge\phi(x,y))$. We define a reduction
of $(q,Q)$ to $(p,P)$ by
\[
x\mapsto f(x):=(x,\phi,q,d,2^{r(|x|)}).
\]
Given a witness $(v,w)$ for $P(f(x),(v,w))$ we get a witness for
$Q(x,v)$ simply by taking the first element from the pair $(v,w)$.
\end{proof}

We are indebted to to Emil Je\v{r}\'abek for the following proposition.

\bpr[E. Je\v{r}\'abek, unpublished]
There exists a complete problem in {\bf TFNP} w.r.t. polynomial reductions if and only if there exists a complete problem in {\bf TFNP} w.r.t. \emph{many-one} polynomial reductions.
\epr

The proposition is an immediate corollary of the following lemma.

\bl
For every {\bf TFNP} problem $P$, there exists a {\bf TFNP} problem $P'$ such that for every  {\bf TFNP} problem $Q$, if $Q$ is polynomially reducible to $P$, then $Q$ is many-one polynomially reducible to $P'$.
\el

\bprf
Let $P$ be given by a polynomial $p$ and a binary relation $R$. 
%(w.l.o.g. we will assume that the polynomial bound is implicit in $R$). 
We define a binary relation $R'(u,v)$ as follows. Interpret a string $u$ as an encoding of a string $x$ and an oracle Boolean circuit $C$. We will only allow oracle circuits that have $p(n)$ input bits for a possible oracle answer for each query of length $n$. 
Then $R'((x,C),v)$ will be defined to be true if $v$ encodes a computation of $C$ on input $x$ with the oracle queries and answers to be pairs $r,s$ such that $R(r,s)$ holds true; in other words, $v$ encodes a computation of $C$ that uses $P$ as an oracle. Furthermore, $R'(u,v)$ is defined to be true, if $u$ does not have the form described above. Clearly, $R'$ defines a total  problem: given $(x,C)$, we can run $C$ on input $x$ using $P$ as an oracle.

Suppose $Q$ is reducible to $P$ using a polynomial time query machine $M$. For each input $x$ for the problem $Q$, we can construct in polynomial time an oracle Boolean circuit $C$ that simulates computations of $M$ on $x$. Given a string $v$ such that $R'((x,C),v)$, we get an output string $y$ of the computation of $M$ that satisfies $Q(x,y)$, because $M$ is a polynomial reduction of $Q$ to $R$. So the reduction is given by the polynomial time functions $x\mapsto (x,C)$ and $v\mapsto y$, where $y$ is the output of the computation encoded by $v$.
\eprf

Furthermore, Je\v{r}\'abek noted that we also get an equivalent conjecture if we use the following modification. Let us denote by $\overline{\bf TFNP}$ the class of search problems defined in the same way as {\bf TFNP} except that the binary relations are only required to be in {\bf NP}.%
\footnote{It would be more logical to use {\bf TFP} for what is
  called {\bf TFNP} and reserve {\bf TFNP} for $\overline{\bf TFNP}$.}
Many-one polynomial reductions for $\overline{\bf TFNP}$ are defined exactly in the same way as for {\bf TFNP}.

\bpr
There exists a complete problem in {\bf TFNP} if and only if there exists a complete problem in $\overline{\bf TFNP}$.
\epr

\prfh
(1) Every problem $P$ in  {\bf TFNP} is, by definition, also in $\overline{\bf TFNP}$. 
(2) Let $Q\in\overline{\bf TFNP}$. Let $Q$ be given by a binary relation  $\exists^p z.R(x,y,z)$. Then the binary relation $R'$ defined by 
\[
R'(x,(y_1,y_2)):= R(x,y_1,y_2)
\]
defines a problem in {\bf TFNP}. 
%Using these two observations as a hint, it is very easy to finish the proof. We leave it to the reader.
\qed

\subsection{Some arguments supporting the conjecture}

It is always difficult to justify a mathematical conjecture. Either the sentence is true, or it is false, but unlike in physics, in mathematics there are no experiments that may support one or the other. Thus the belief in a conjecture is based on subjective feelings. 
Here are our reasons why we believe that the conjecture should be true.

\ben
\item Every {\bf TFNP} problem is based on some mathematical principle
  that ensures that for every input there exists a solution. Although
  these principles are simple for the basic classes of {\bf TFNP}
  problems, it seems likely that there is no universal mathematical
  principle that would work for every {\bf TFNP} problem.
\item Combinatorial characterizations of provably total polynomial search problems have been obtained for some fragments of Bounded Arithmetic. 
 The  description of these combinatorial problems suggests that their strength
  increases with increasing strength of the theories.%
\footnote{We only hypothesize that the strength of fragments $S^i_2$ of Bounded Arithmetic increases with increasing~$i$, but this hypothesis is supported by a connection with the Polynomial Hierarchy in computational complexity~\cite{kpt}.}
\item  An oracle has been
  constructed relative to which the conjecture holds true~\cite{herbr}. 
\item The connection with search problems verifying the consistency
  of a theory that we describe below can also be viewed as a
  supporting argument.
\een

\subsection{Herbrand Consistency Search}

Conjecture~{\sf TFNP} has another equivalent form in which the concept
of consistency plays a key role. The well-known Herbrand theorem provides a ``combinatorial'' characterization of provability in predicate calculus (see, e.g., \cite{buss-handbook}). In particular one can characterize the consistency of theories. Let us consider logic without equality and the special case of universal sentences.
According to Herbrand's theorem a
universal sentence $\Phi:=\forall x_1\dots\forall x_k.\phi(x_1\dts x_k)$, where $\phi$ is quantifier free, is
consistent if and only if for every family of terms $\tau_{ij}$,
$i=1\dts n$, $j=1\dts k,$
\bel{e-herbrand}
\bigwedge_{i=1}^n\phi(\tau_{i1}\dts \tau_{ik})
\ee
is satisfiable as a propositional formula. Sentences expressing consistency using standard proofs (in Hilbert-style, or Gentzen calculi) and sentences expressing consistency using Herbrand's theorem are equivalent provably in every theory that proves Herbrand's theorem. However, the corresponding restricted finite versions of consistency statements are essentially different because the transformation of standard proofs into sets of terms that witness provability in Herbrand's theorem is nonelementary. But here we are interested in a different aspect of Herbrand's theorem: the complexity of finding a satisfying assignment for~(\ref{e-herbrand}). Since the formula, as a proposition, is always satisfiable when $\Phi$ is consistent, every consistent
universal sentence defines a natural {\bf TFNP} problem.

\begin{definition}
Let $\Phi:=\forall x_1\dts x_k.\phi(x_1\dts x_k)$ be a consistent universal sentence.  Then $HCS(\Phi)$, the \emph{Herbrand Consistency Search for $\Phi$}, is the following total polynomial search problem. Given terms $\tau_{ij}$ in the language of $\Phi$,  $i=1\dts n$, $j=1\dts k$, find a truth assignment to the atomic  subformulas occurring in $\phi(\tau_{i1}\dts\tau_{ik})$, for $i=1\dts n$, that makes $\bigwedge_{i=1}^n \phi(\tau_{i1}\dts\tau_{ik})$ true.
\end{definition}

For simplicity, we define Herbrand consistency search only for
universal sentences in this paper, but  using Skolemization, one can easily extend this
definition to conjunctions of prenex formulas.  In
\cite{herbr} we proved the following theorem.

\begin{theorem}\label{t-hcon}
  For every total polynomial search problem $P$, there exists a
  consistent universal sentence $\Phi$ such that the problem $P$ is
  many-one polynomially reducible to $HCS(\Phi)$.
\end{theorem}

Using this theorem we can state Conjecture~{\sf TFNP} in the following
equivalent form.

\begin{conjecture}{equivalent to {\sf TFNP}}
For every theory $T\in\cal T$ there exists a consistent universal
sentence $\Phi$ such that $HCS(\Phi)$ is not polynomially reducible to
any {\bf TFNP} problem provably total in $T$, i.e.,
$HCS(\Phi)\not\in{\bf TFNP^*}(T)$.
\end{conjecture}

%% As with Conjecture~{\sf CON$^N$}, one can ask how much stronger one theory must be than another in order to be able to prove the totality of more polynomial
%% search problems. But we can also ask: what is a search problem whose
%% totality is not provable in a given theory? The following could be an answer to
%% both questions.

This form of the {\sf TFNP} conjecture suggests a natural question: what is a sentence $\Phi$ that is likely not in ${\bf TFNP^*}(T)$? The following could be an answer to this question.

\begin{conjecture}{\sf TFNP$^+$}
Suppose $T\in\cal T$ is axiomatized by a universal 
sentence. Then $T$ does not prove that $HCS(T)$ is total for any
formalization of it by a $\Delta^b_1$ formula. 
% In symbols, $HCS(T)\not\in{\bf TFNP^*}(T)$. % not equivalent!!!
\end{conjecture}

Note that if $T$ is strong enough to prove Herbrand's theorem, then it
does not prove the totality of $HCS(T)$ formalized in a natural way,
because if it did, it would prove its own consistency. However, this
does not exclude the possibility that it proves the totality for some
contrived definition. Although we call it a conjecture, we are not
very confident that it is true. But suppose it were true and suppose
that $S\in\cal T$ is axiomatized by a universal formula and $T$ is a theory that proves Herbrand's Theorem and the consistency of $S$. Then we would have $HCS(S)\in{\bf TFNP^*}(T)\setminus{\bf
  TFNP^*}(S)$. Thus according to this conjecture, adding the
consistency of a theory to itself produces more provably total
polynomial search problems (at least for theories axiomatized by a universal formula).

\section{Propositional proof systems, disjoint {\bf NP}-pairs and
  disjoint {\bf coNP}-pairs}\label{sec5}

So far we were concerned with first order theories. In this section we
will show that one can also use other formal systems, namely,
propositional proof systems, in order to state and study conjectures
about incompleteness in the finite domain.

Let a language for classical propositional logic be fixed; say, we
take connectives $\neg,\wedge,\vee$ and variables $p_1,p_2,\dots$. Let
{\it TAUT} be the set of all tautologies and {\it SAT} be the set of
all satisfiable propositions.
Following~\cite{cook-reckhow}, we say that \emph{a proof system is a
  polynomial time computable function $P$ from $\Sigma^*$ onto {\bf
    TAUT}}.%
\footnote{Recall that in this paper  
$\Sigma$ denotes $\{0,1\}$, but in this definition it could be any finite alphabet of size at least~2.} 
If $P(w)=\phi$, we say that $w$ is a proof of $\phi$ in
the proof system $P$. This elegant definition captures three basic
properties of proof systems: \ben
\item the relation \emph{``$w$ is a proof of $\phi$''} is decidable in
  polynomial time;
\item the system is sound;
\item the system is complete.
\een
In the rest of this section the term ``proof system'' will
always refer to ``\emph{propositional} proof system''.

According to this definition, a proof can be any evidence that shows
logical validity of a proposition. %% For example, we may take a first
%% order theory $T\in\cal T$, fix a formula $\tau(x)$ that represents
%% {\it TAUT} in $T$, and then proclaim any $T$-proof of $\tau(\bar\phi)$ to be
%% a proof of $\phi$. % repreated below
The standard formalizations of propositional
calculus based on axioms and logical rules are systems from a special
class of proof systems, called \emph{Frege systems.}

We say that \emph{a proof system $P$ is polynomially bounded} if there
exists a polynomial $p$ such that every tautology $\phi$ has a
$P$-proof of length at most $p(|\phi|)$. Since {\it TAUT} is {\bf
  coNP}-complete, the existence of a polynomially bounded proof system
is equivalent to {\bf NP}={\bf coNP}.

A weaker concept is length optimality. We say that \emph{a proof
  system $P$ is length-optimal} if for every proof
system $Q$, there exists a polynomial $p$ such that if $\phi$ has a
proof of length $n$ in $P$, then it has a proof of length at most
$p(n)$ in $Q$. (Length-optimality is a nonuniform version of p-optimality that will be defined in Section~\ref{sec6}.)  
In~\cite{KP} we showed that Conjecture~{\sf CON}$^N$ is
equivalent to the following one.

\begin{conjecture}{equivalent to {\sf CON}$^N$}
There exists no length-optimal proof system.
\end{conjecture}

Why do we believe that this conjecture is true? An argument that we can give is based on a construction of proof systems used to prove that the two statements of Conjecture~{\sf CON}$^N$ are equivalent. Given an arithmetical theory $T$, we can formalize the concept of a propositional tautology by some formula $\tau(x)$. For a given tautology $t$, we take its G\"odel number $n$ and treat any first order $T$-proof of $\tau(\bar{n})$ as a proof in a propositional proof system. Then it seem plausible that in stronger theories we can prove some tautologies by shorter proofs. Moreover, one can show that these proof systems are in a sense universal. So the fact that the logical strength of theories cannot be bounded is likely to be projected into these proof systems.

Another argument supporting the conjecture is from our experience with specific proof systems studied in proof complexity. Most systems are based on some class of formulas and deduction rules. If we enlarge the class of formulas then, usually, the system becomes stronger. For example, if we use quantified Boolean formulas instead of ordinary Boolean formulas, the system seems much stronger. For some weak systems, in particular, bounded depth Frege systems, this has actually been proven~\cite{Impagliazzo-Krajicek}. As, apparently, there is no limit on how strong expressive power formulas can have, we also believe that there is no limit on how efficient a proof system can be.

%% {\bf Why do we believe? Every theory gives a proof system, stronger theories give stronger. This conjecture is a complexity theoretical version of incompleteness - first order proofs are replaced by propositional proofs, provability is replaced by provability by polynomial size proofs. Connections with bounded arithmetic show that it is a nonuniform version of unprovability.}

\subsection{Disjoint {\bf NP} pairs}

In \cite{Razborov94} Razborov defined the \emph{canonical
  pair of a proof system} $P$ to be the pair of sets $(PR(P),NSAT^*)$
where $PR(P)=\{(\phi,2^m);\ \phi\mbox{ has a $P$-proof of length at
  most }m\}$, and  $NSAT^*=\{(\phi,2^m);\ \neg\phi\mbox{ is satisfiable
}\}$. Note that it is a pair of two disjoint {\bf NP} sets. If a proof
system $P$ simulates a proof system $Q$, then  $(PR(Q),NSAT^*)$ is
polynomially reducible to  $(PR(P),NSAT^*)$ in the following sense.

We say that \emph{a disjoint {\bf NP} pair $(A,B)$ is polynomially reducible
to  a disjoint {\bf NP} pair $(C,D)$} if there exists a polynomial
time computable function $f$ that maps $A$ into $C$ and $B$ into $D$. We say that pairs $(A,B)$ and $(C,D)$ are \emph{polynomially equivalent} if there are reductions between them in both directions.

It is not difficult to show that canonical pairs of proof systems are
universal in the class of all disjoint {\bf NP} pairs, which means
that every disjoint {\bf NP} pair $(A,B)$ is polynomially reducible to
the canonical pair of some proof system $P$. In fact, even more is true.

\begin{proposition}[\cite{GSZ}]
For every disjoint {\bf NP} pair $(A,B)$, there exists a proof system whose
canonical pair is polynomially \emph{equivalent} to $(A,B)$.
\end{proposition}

% The idea of the proof is
% to take any proof system and add the tautologies that express that
% $A\cap B\cap\{ 0,1\}^n=\emptyset$ for $n=1,2,\dots$. Observe that
% there exists a polynomial $p$ such that if $x\in B$, $|x|=n$, then the
% proposition $\alpha_n(w)$ expressing $w\not\in A\cap\{ 0,1\}^n$ has a
% $P$-proof of length at most $n$, and
% the same with $A$ and $B$ swtiched. So $w\mapsto\alpha_n(w)$ is a
% polynomial reduction of $(A,B)$ to the canonical pair of~$P$.  

Furthermore, if $P$ and $Q$ are proof systems and there exists a polynomial
$p$ such that for every tautology $\phi$, if $\phi$ has a $P$-proof of
length $n$, then $\phi$ has a $Q$-proof of length at most $p(n)$, then
the canonical pair of $P$ is polynomially reducible to the canonical
pair of~$Q$. Indeed, the mapping $(\phi,2^n)\mapsto (\phi,2^{p(n)})$
is such a reduction. Thus we get:

\begin{proposition}[\cite{Razborov94,KMT}]\label{disjNP-CON-N}
If $P$ is a length-optimal proof system, then its canonical pair is a complete
disjoint {\bf NP} pair with respect to polynomial reductions
(i.e., every disjoint {\bf NP} pair is reducible to it).\footnote{Razborov proved this fact for p-optimal proof systems (see Definition~\ref{def8} below); K\"obler, Messner and Tor\'an improved it to length optimal proof systems.}
\end{proposition}

Therefore the following conjecture is a strengthening of
Conjecture~{\sf CON$^N$}.

\begin{conjecture}{\sf DisjNP}
  There exist no complete disjoint {\bf NP} pair (with respect to
  polynomial reductions).
\end{conjecture}

Gla{\ss}er et al.~\cite{GSSZ} constructed an oracle relative to which there
is no complete disjoint {\bf NP}-pair. Other than that, we have little
supporting evidence. A combinatorial characterization of the canonical
pair has only been found for the resolution proof system. 
In~\cite{GSSZ} they also constructed an oracle relative to which there
exists a complete disjoint {\bf NP}-pair, but no length-optimal proof system
exists, i.e., Conjecture~{\sf DisjNP} fails, but Conjecture~{\sf CON$^N$}
holds true.

\subsection{Disjoint {\bf coNP} pairs}

We now turn to disjoint {\bf coNP} pairs. When comparing different
disjoint {\bf coNP}-pairs, one can use the same
polynomial reduction as used for disjoint {\bf NP}-pairs; 
hence one can also ask similar questions. In particular, are there
disjoint {\bf coNP} pairs inseparable by a set in {\bf P}? Are there
complete disjoint {\bf coNP} pairs? We believe that the answer to the
first question is yes, because we accept ${\bf NP\cap coNP}\neq{\bf P}$
as a very likely fact. The answer to the second question is less
clear, but we still lean to the negative answer.

\begin{conjecture}{\sf DisjCoNP}
  There exist no complete disjoint {\bf coNP} pair (with respect to
  polynomial reductions).
\end{conjecture}

The next proposition states that 
Conjecture~{\sf TFNP} is a consequence of the above conjecture. 

\bpr\label{TFNP-coNP}
If there exists a complete {\bf TFNP} problem, then there exists a
complete disjoint {\bf coNP} pair.
\epr

The proposition follows from the two lemmas below. First we need a definition.

\begin{definition}
Let a {\bf TFNP} problem $(p,R)$ be
given. Assume that $R(x,y)\tto |y|=p(|x|)$. The \emph{canonical disjoint {\bf coNP} pair of $(p,R)$} is the pair  $(A_0,A_1)$ defined as
follows. The elements of $A_0\cup A_1$ are pairs $(x,C)$ where $x$
is an arbitrary binary string and $C$ is a Boolean circuit with
$p(|x|)$ bit-inputs and one bit-output. The sets $A_0$ and $A_1$ are defined
by
\bel{canonical-coNP}
(x,C)\in A_i\ \equiv\ \forall y (R(x,y)\to C(y)= i).
\ee
\end{definition}
The condition that, for a given $x$, all elements $y$ satisfying $R(x,y)$ have the same length is, clearly, not essential, because we can always pad the string $y$ to the maximal length $p(|x|)$.

\bl
For every disjoint {\bf coNP} pair  $(B_0,B_1)$ there exists a {\bf TFNP} problem $(p,R)$ such that  $(B_0,B_1)$ is polynomially reducible to the canonical disjoint {\bf coNP} pair of $(p,R)$.
\el

\bprf
Let a disjoint {\bf coNP} pair $(B_0,B_1)$ be given. Suppose that $B_i$s are  defined by
\[
x\in B_i\ \equiv\ \forall y(|y|\leq r_i(|x|)\to\beta_i(x,y))
\]
for $i=0,1$, where $\beta_i$ is computable in polynomial time and
$r_i$ is a polynomial. Let the
binary relation $R$ be defined by
\[
R(x,z)\ \equiv\ \exists i\in\{0,1\}\exists y
(z=(i,y)\wedge |y|\leq r_i(|x|)\wedge \neg\beta_i(x,y)).
\]
Since $\beta_i$s are computable in polynomial time, so is also $R$ and
the length of every $z$ satisfying $R(x,z)$ is polynomially bounded in
the length of $x$. Furthermore, since $B_0$ and $B_1$ are disjoint,
$R$ is total. Again, by suitably padding $z$ we may ensure that  
$R(x,z)\tto |z|=p(|x|)$ for some polynomial $p$. 
Let $(A_0,A_1)$ be the canonical pair of $(p,R)$. 
The pair $(B_0,B_1)$ is reducible to $(A_0,A_1)$  by the mapping
\[
x\ \mapsto\ (x,C),
\]
where $C$ is a circuit such that $C(i,y)=1-i$, because for this $C$, $(x,C)\in A_j$ iff $x\in B_j$.
\eprf

\bl
Let $(p,P)$ and $(q,Q)$ be two {\bf TFNP} problems such that $R(x,y)\tto |y|=p(|x|)$ and $Q(x,y)\tto |y|=q(|x|)$. Let $(A_0,A_1)$
respectively $(B_0,B_1)$ be their canonical {\bf coNP} pairs and
suppose that 
$(p,P)$ is polynomially many-one reducible to $(q,Q)$. Then $(A_0,A_1)$
is reducible to $(B_0,B_1)$.
\el
\bprf
Let $(p,P)$, $(q,Q)$ and a polynomial many-one reduction $(f,g)$ of
$(p,P)$ to $(q,Q)$ be given.  Let $(A_0,A_1)$ and $(B_0,B_1)$ be the
canonical {\bf coNP} pairs of $(p,P)$ and $(q,Q)$. We define a
polynomial reduction of $(A_0,A_1)$ to $(B_0,B_1)$ as follows. For an
input of the form $(x,C)$ where $C$ is a Boolean circuit, we put
\[
h(x,C)=(f(x),D_x),
\]
where $D_x$ is a Boolean circuit with $q(|f(x)|)$ bit inputs such that for all
$y$ of length $q(|f(x)|)$,
\bel{e54}
D_x(y) = C(g(x,y)).
\ee
If an input $z$ does not have the required form, we put $h(z)=0$. We will check that this defines a polynomial reduction of  $(A_0,A_1)$ to $(B_0,B_1)$. Let $(x,C)\in A_i$ and let $y$ be any number such that $|y|=q(|f(x)|)$ and $Q(f(x),y)$. Since $P(x,g(x,y))$, we have $C(g(x,y))=i$ by the definition of $A_i$. By (\ref{e54}), $D_x(y)=i$. This proves that $f(x)\in B_i$. 
\eprf

\subsection{Multivalued functions}

A class closely related to {\bf TFNP} and the question whether there
exists a complete problem in this class were studied by Beyersdorff, K\"obler and Messner~\cite{BKM}. We need a couple of preliminary definitions.

A multivalued partial function $f$ is called an \emph{{\bf NP}
  multivalued function} if it is computed by a nondeterministic
polynomial time Turing machine $M$ in the following sense. $M$ stops
in two possible states: ACCEPT and REJECT. For a given input value $x$,
the values of $f$ are those words on the output tape which appear when
the state ACCEPT is reached. For a function $f\in{\bf NPMV}$ we denote
by $f\{ x \}$ the set of all values for the input $x$. Thus $f$ is total iff $f\{x\}\neq\emptyset$ for all $x$. 
The class of {\bf NP} \emph{multivalued
functions} is denoted by {\bf NPMV}.  The class of \emph{total} {\bf
  NP} \emph{multivalued functions} is denoted by {\bf NPMV}$_t$. 

By their nature, {\bf NPMV}$_t$ functions are {\bf TFNP}
problems, but there is an essential 
difference in how one defines reduction. For $f,g\in{\bf
NPMV}$, we say that $f$ is polynomially reducible to $g$ if there
exists a polynomial time computable function $h$ such that for all $x$,
\[
f\{x\}=g\{h(x)\} .
\] 

A relation to our Conjecture~{\sf TFNP} is given by the following
proposition. 
 
\bpr\label{TFNP-NPMV}
The existence of a complete function in {\bf NPMV}$_t$
implies the existence of a complete {\bf TFNP} problem. 
\epr
\bprf Let
$g$ be a complete function in {\bf NPMV}$_t$. We can represent $g$
using a polynomial time computable ternary relation as follows.
\[
g\{x\}=\{y;\ \exists^p z.R(x,y,z)\}.
\]
Recall that the superscript at the existential quantifier means that we tacitly assume that there exists a polynomial bound $p$ such that
$R(x,y,z)$ is satisfied only if the lengths of $y$ and $z$ are bounded 
by $p(|x|)$. Define
\[
Q(x,u):= R(x,(u)_1,(u)_2).
\]
We claim that $Q$ defines a complete {\bf TFNP} problem. Let $S(x,y)$
be a binary relation computable in polynomial time viewed as a {\bf
  TFNP} problem (again, we tacitly assume an implicit polynomial bound on the
length of $y$). Define a function $f\in \mbox{\bf NPMV}_t$ by
\[
f\{x\}:=\{y;\ S(x,y)\}.
\]
Since $f$ is reducible to the complete function $g$, there exists a
polynomial time computable function $h$ such that $f\{x\}=g\{h(x)\}$,
which is equivalent to
\[
\{y;\ S(x,y)\}=\{y;\ \exists z.R(h(x),y,z)\}=\{y;\ \exists z.Q(h(x),(y,z))\}.
\]
Thus the pair of functions $h,k$, where $k(u):=(u)_1$, is a polynomial
reduction of $S$ to $Q$.
\eprf

We do not know if the opposite implication holds true. 
Beyersdorff et al. \cite{BKM} %(Theorem 9) 
proved that if there exists a complete function in {\bf NPMV}$_t$, then there
exists a complete disjoint {\bf coNP} pair. This is now a consequence
of Propositions \ref{TFNP-coNP} and  \ref{TFNP-NPMV}.

\section{Classification of conjectures}\label{sec6}

%% In this section we propose some classification of the conjectures
%% considered so far. This should be viewed as a preliminary proposal,
%% because more research is needed to find clear and natural criteria for
%% classification. But first we need to introduce a few more conjectures.

\subsection{Uniform and nonuniform}

A more natural way to compare proof systems than just comparing the
lengths of proofs is polynomial simulation. This is a concept,
introduced in~\cite{cook75}, is similar to polynomial reductions used in
the theory of {\bf NP}-completeness and those we used to compare {\bf
  TFNP} problems.

\begin{definition}\label{def8}
We say that \emph{a proof $P$ system polynomially simulates a proof
  system $Q$} if there exists a polynomial time computable function
such that given a $Q$-proof $d$ of $\phi$, $f(d)$ is a $P$-proof of
(the same) $\phi$.  We say that \emph{a proof system $P$ is p-optimal}
if it polynomially simulates every proof system.
\end{definition}

Using this concept we can state a conjecture slightly weaker than Conjecture~{\sf CON$^N$}.

\begin{conjecture}{\sf CON}
There exists no p-optimal proof system.
\end{conjecture}
 
In~\cite{KP} we proved that this conjecture is equivalent to the
following uniform version of Conjecture~{\sf CON$^N$}.

\begin{conjecture}{\sf equivalent to {\sf CON}}
  For every $S\in\cal T$, there exists $T\in\cal T$ such that
  $S$-proofs of $Con_T(\bar n)$ cannot be constructed in polynomial time in $n$.
\end{conjecture}

A uniform version of Conjecture~{\sf RFN$^N_1$} is obtained in the same
way.

\begin{conjecture}{\sf RFN$_1$}
  For every $S\in\cal T$, there exists $T\in\cal T$ such that
  $S$-proofs of $\Sigma_1^bRFN_T(\bar n)$ cannot be constructed in
  polynomial time in $n$.
\end{conjecture}

Except for modifications of these conjectures, such as Conjecture~{\sf CON$^{N+}$}, we do not know of any other pair of uniform and nonuniform conjectures. In particular, {\sf TFNP} is apparently uniform, but we do not know if it has a nonuniform companion.

Note that ${\bf NP}\neq{\bf coNP}$ is implied by the nonuniform conjectures {\sf CON$^N$} and {\sf RFN$^N_1$}, while the uniform versions {\sf CON} and {\sf RFN$_1$} are only known to imply ${\bf P}\neq{\bf NP}$. Thus we should also classify ${\bf NP}\neq{\bf coNP}$ as nonuniform and ${\bf P}\neq{\bf NP}$ as uniform.
Then it may seem strange that according to this classification ${\bf NP}\neq{\bf coNP}$ should be a nonuniform conjecture, in spite of the fact that both {\bf NP} and {\bf coNP} are uniform complexity classes.
But if we look at  ${\bf NP}\neq{\bf coNP}$ from the point of view of proof complexity, then it is clearly a nonuniform version of ${\bf P}\neq{\bf NP}$. Just consider the following equivalent formulations of these conjectures:
\bi
\item  ${\bf P}\neq{\bf NP}$ $\Leftrightarrow$ there exists a proof system $P$ such that for every tautology $\tau$ a $P$-proof of $\tau$ \emph{can be constructed in polynomial time;}
\item ${\bf NP}\neq{\bf coNP}$ $\Leftrightarrow$ there exists a proof system $P$ such that  every tautology $\tau$ \emph{has a $P$-proof of polynomial length.}
\ei
However, although Conjecture~{\sf DisjNP} seems to be uniform, it does imply the nonuniform Conjecture~{\sf CON$^N$} (see Proposition~\ref{disjNP-CON-N}). We do not have an explanation for this.

\subsection{Logical complexity}

We started with statements about finite consistency, statements that
express facts about logic, and eventually arrived at statements about
disjoint sets of certain complexity, statements from 
structural complexity theory that apparently have nothing to do with
the main theme of incompleteness. But one should realize that
expressing these conjectures using concepts from computational
complexity theory is just a convenient way to state them. It seems that it should be possible  to present all uniform conjectures as statements about unprovability of
certain sentences in theories from the class $\cal T$. The following proposition shows how to state Conjecture~{\sf CON} in this way.

\bpr\label{pr-6.1}
There exists a $p$-optimal proof system (for {\it TAUT}) if and only if there exists a
theory $T\in\cal T$ such that for every proof system $P$ there exists
a definition of $P$ by a $\Delta^b_1$ formula 
such that $T$ proves the soundness of $P$ represented by this formula.
\epr
For the proof, see \cite{kniha}, pages 578-9.
Next proposition shows how to express Conjecture~{\sf DisjNP} as a statement about unprovability of certain sentences.

\bpr\label{pr-6.2}
There exists a complete disjoint {\bf NP} pair if and only if there exists
a theory $T\in\cal T$ such that for every disjoint {\bf NP} pair
$(B_0,B_1)$ there are $\Sigma^b_1$ definitions of $B_0$ and $B_1$ for
which $T$ proves that they define disjoint sets.
\epr
\bprf
Suppose that there exists a complete disjoint {\bf NP} pair $(A_0,A_1)$. Let
$\exists^py.\alpha_i(x,y)$ be $\Sigma^b_1$ definitions of $A_i$,
$i=0,1$. Define a theory $T$ to be 
%$S^1_2$ plus an axiom expressing that $A_0\cap A_1=\emptyset$
\[
S^1_2\ +\ \forall x(\neg\exists^py.\alpha_0(x,y)\vee\neg\exists^py.\alpha_1(x,y)).
\]
Let $(B_0,B_1)$ be an arbitrary disjoint {\bf NP} pair. Let
$\exists^py.\beta_i(x,y)$ be some $\Sigma^b_1$ definitions of $B_i$,
$i=0,1$. Since  $(A_0,A_1)$ is complete, there exists a polynomial
time reduction $f$ of $(B_0,B_1)$ to $(A_0,A_1)$. Consider the
following definitions of $B_i$, $i=0,1$, by $\Sigma^b_1$ formulas:
\[
\exists^py.\beta_i(x,y)\wedge\exists^pz.\alpha_i(f(x),z).
\]
It is clear that they define the sets  $B_i$ correctly and that $T$
proves that sets defined by these formulas are disjoint.

The proof of the converse implication is a standard 
argument that we have already presented in the proof of
Lemma~\ref{l-4.1}, so we will be very brief.

Let $T$ be a theory with the property stated in the proposition. For
$i=0,1$, let $A_i$ be the set of tuples $(x,\beta_0,\beta_1,d,a)$ such
that 
\bi
\item $\beta_0$ and $\beta_1$ are $\Sigma^b_1$ formulas, $d$ is a
$T$-proof of the disjointness of the sets defined by $\beta_0$ and $\beta_1$,
$a$ is a nondeterministic time bound for $\beta_0$ and $\beta_1$, and
$\exists^py.\beta_i(x,y)$ holds true.
\ei
We leave to the reader to verify that these conditions define a
disjoint {\bf NP} pair and that every disjoint {\bf NP} pair is
polynomially reducible to it.
\eprf
The non-existence of a complete disjoint {\bf coNP} pair, Conjecture~{\sf DisjCoNP}, can be expressed as a statement about provability in the same way. Conjecture~{\sf TFNP} was, in fact, introduced as a sentence about unprovability in theories in $\cal T$. 

Thus a natural way to classify such conjectures is according to the logical complexity of sentences that are claimed to be unprovable. The two most important classes are $\forall\Pi^b_1$ and $\forall\Sigma^b_1$ (i.e., the sentences of the form: universally quantified $\Pi^b_1$ and $\Sigma^b_1$ formulas). Our uniform conjectures are classified as follows:

\medskip
$\forall\Pi^b_1$ -- {\sf CON, DisjNP};

\medskip
$\forall\Sigma^b_1$ -- {\sf RFN$_1$, TFNP, DisjCoNP}.

\medskip

\subsection{Some related statements}

Several concepts related to our conjectures have been studied. We will present some of these sentences here. We will call them conjectures, since we believe that they are true, but we do not have essentially any supporting argument for their truth.

We have observed that Conjecture~{\sf CON$^N$} can be strengthened to Conjecture~{\sf DisjNP}. Its uniform version, Conjecture~{\sf CON}, can, furthermore, be strengthened in a different way. Recall that {\bf UP}, \emph{unambiguous} {\bf P}, is the class of languages that are accepted by polynomial time \emph{nondeterministic} Turing machines that satisfy the property that for every accepted input, there is a \emph{unique} accepting computation.  K\"obler, Messner and Tor\'an~\cite{KMT} proved that if there exists a p-optimal proof system, then {\bf UP} has a complete set with respect to many-one reductions. Hence the following is a strengthening of Conjecture~{\sf CON}.

\begin{conjecture}{\sf UP}
There is no complete set, with respect to many-one reductions, in {\bf UP}.
\end{conjecture}

So far we only talked about proof systems for {\it TAUT}. In the same way one
can define proof systems and polynomial simulations for any set. In
particular, a \emph{proof system for {\it SAT}} is a polynomial time
computable function from $\Sigma^*$ onto {\it SAT}. There is one
essential difference between proof systems for {\it TAUT} and {\it
  SAT}---the latter does have polynomially bounded proof systems. In
fact, the definition of {\it SAT} itself gives one such proof system; in this
system any pair $(\phi,a)$, where $a$ is a satisfying assignment of a formula $\phi$ is a proof (of the
satisfiability of) $\phi$. This is called the \emph{standard} proof
system for {\it SAT}.

Here is an example of a nonstandard proof system $P$ for {\it SAT}. In
$P$ a proof of $\phi$ is either a pair $(\phi,a)$, where $a$ is a  satisfying assignment of $\phi$, or it is
$\phi$ itself in the case when $\phi$ is a proposition $\gamma_n$
expressing, in a natural way, the fact that $n$ is a composite 
number and $n$ is a composite number. Note that in the standard proof system the proof of $\gamma_n$
encodes a nontrivial factor of $n$. Hence, if the standard proof
system p-simulated $P$, then factoring would be in polynomial time.

Beyersdorff et al.~\cite{BKM} proved that the existence of a p-optimal
proof system for {\it SAT} implies the existence of a complete
function in {\bf NPMV}$_t$. 
Hence, by our Proposition~\ref{TFNP-NPMV}, it also implies the existence
of a complete problem in {\bf TFNP}. 
To put the conjecture about complete sets
in {\it SAT} into context, we need the following proposition.

\bpr
Let $S\in\cal T$ be a theory such that for every 
theory $T\in\cal T$, $S$-proofs of $\Sigma_1^bRFN_T(\bar n)$ can be
constructed in polynomial time in $n$. 
Then there exists a p-optimal proof system for {\it SAT}.
\epr

\begin{proof}
  Let $sat(x,y)$ be a $\Delta^b_1$ formula expressing the fact that
  $y$ is a satisfying assignment of a propositional formula $x$.  Suppose
  that $S$ satisfies the assumption of the proposition. We define a
  proof system $P$ for {\it SAT} by:
\[
\mbox{$y$ is a $P$-proof of $x$ $\Leftrightarrow$ $y$ is an $S$-proof
  of $\exists z.sat(\bar{x},z)$.} 
\]
Given a proof system $f$ for {\it SAT}, we take $T\in\cal T$ such that
it proves the soundness of $f$, i.e.,
\bel{e-sound}
T\vdash\ \forall y\exists z.sat(f(y),z).
\ee
By Corollary~\ref{l-rfn}, $S$-proofs of $\exists z.sat(f(\bar d),z)$
can be constructed in polynomial time for every~$d$. Thus, given an
$f$-proof $d$ of $f(d)$, we can construct in polynomial time a proof
in~$P$. Hence $P$ is a p-optimal proof system for {\it SAT}.
\end{proof}

Conjectures {\sf DisjNP, DisjCoNP} and {\sf UP} are related to our main conjectures {\sf CON} and {\sf TFNP}. Here is an example of a plausible conjecture that is apparently incomparable with {\sf CON} and {\sf TFNP}.

\begin{conjecture}{\sf NP$\cap$coNP}
There is no complete set in ${\bf NP\cap coNP}$.
\end{conjecture}

Beyersdorff et al.~\cite{BKM} proved that if both {\it TAUT} and {\it
  SAT} have p-optimal proof systems, then there exists a complete set in
${\bf NP\cap coNP}$. Hence Conjecture~{\sf NP$\cap$coNP} is above
Conjecture~{\sf RFN$_1$}. 

The implications between the most important uniform conjectures considered in this paper are
depicted in the figure below. Recall that {\sf CON} is equivalent to the nonexistence of a p-optimal proof system for {\it TAUT}.

{\footnotesize%\scriptsize
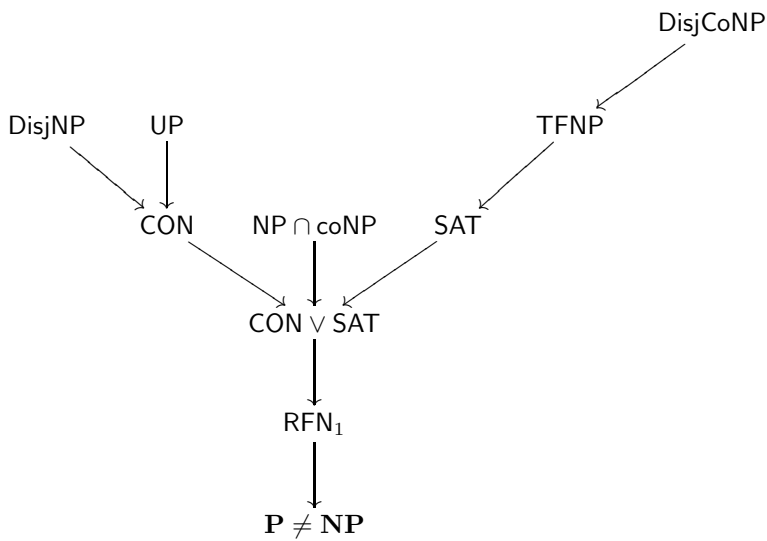
\[\xymatrix@C=5mm{            % @C = column spacing @R = row spacing
  &&&&& {\sf DisjCoNP}\ar[ld]\\
 {\sf DisjNP}\ar[rd]&{\sf UP}\ar[d]& & & {\sf TFNP}\ar[ld]&\\
 & {\sf CON} \ar[rd]&{\sf NP}\cap{\sf coNP}\ar[d] & {\sf SAT}\ar[ld]&&\\
 & & {\sf CON}\vee{\sf SAT}\ar[d]&&&\\
& & {\sf RFN}_1\ar[d]&&&\\
& &{\bf P}\neq{\bf NP}&&&
}\]
}

\subsection{Towards general conjectures}\label{sec-general}

We will focus on uniform conjectures, because the situation there seems to be clearer.
We have seen that our uniform conjectures are statements about unprovability of particular
sentences. The structure of these sentences is determined by
\ben
\item some class $\cal C$ of sentences 
\item associated with computational
  problems $\cal P$, and 
\item some complexity hierarchy $\cal H$ of the associated problems.
\een
The conjectures say that the more difficult the associated
computational problem is, the more difficult is to prove the
sentence. 

Consider, for example, Conjecture~{\sf DisjNP}. In this conjecture we have 
sentences expressing that two sets defined by
$\Sigma^b_1$ sentences are disjoint. These sentences are of the form: 
\bel{e-d}
\forall x(\neg\phi(x)\vee\neg\psi(x)),
\ee
where $\phi$ and $\psi$ are $\Sigma^b_1$ sentences defining the two sets. These sentences are equivalent to universally quantified $\Pi^b_1$ sentence, but they are just some specific universally quantified $\Pi^b_1$ sentences. For such sentences, a natural task is, for a given $x$, to decide which of the two $\neg\phi(x)$ or $\neg\psi(x)$ is true. The complexity hierarchy of the computational problems is defined using polynomial time reductions. 

Consider Conjecture~{\sf CON}. In the equivalent form of this conjecture given by Proposition~\ref{pr-6.1}, the sentences expressing that a propositional proof system $P$ is sound are also universally quantified $\Pi^b_1$ sentences. They have the form
\bel{e-f}
\forall x,y,z({ proof}_P(x,y)\to sat(x,z)),
\ee
where $proof_P(x,y)$ is a $\Delta^b_1$ formula expressing that $y$ is
a $P$-proof of $x$. The structure of sentences (\ref{e-d}) and (\ref{e-f}) is similar (essentially, they are universally quantified disjunctions), but
the length of $y$ in the second formula is not polynomially bounded in the length of $x$.  
% We associate with these formulas the same computational task, but 
Furthermore, we use a different kind or reductions to define the hierarchy: in the first case, a reduction can map $x$ to another element, but we do not care about the witnesses of the $\Sigma^b_1$ formulas; in the
second case, $x$ does not change, but we map a witness $y$ to another
witness.  

Ideally, we would like to state a general conjecture from which our current conjectures would follow as special cases. However, to be able to do that, we first need to fully understand
what are the classes $\cal C$ whose sentences can be associated with computational tasks,
what are the computational problems $\cal P$, and 
what are the complexity hierarchies $\cal H$.
So far we only have examples.

%% What about other universally
%% quantified $\Pi^b_1$ sentences? Can we always find 
%% some natural
%% computational task associated with them?

%% But there is an essentially different way of viewing these
%% statements: we can view them as new axioms, e.g., like large cardinal axioms
%% in set theory. If we view them in this way, then tere is no hope to
%% find one general principle.

%% Instead of looking for the strongest conjectures, it may be better to look for a general ``metaconjecture'' that would
%% specify for which classes $\cal C$, problems $\cal P$ and hierarchies
%% $\cal H$ such statements should be true.

\section{The role of reductions}\label{sec8}

In Propositions~\ref{pr-6.1} and~\ref{pr-6.2} we saw that conjectures whose statements used reductions can be equivalently stated without referring to any concept of polynomial reduction. In this section we will explain how polynomial reductions naturally appear when we compare the logical strength of sentences.

When we are comparing sentences from some class $\cal C$, we do it
with respect to a specific base theory $T$. Thus for some $\Phi,\Psi\in\cal
C$, we are asking whether $T\vdash\Phi\to\Psi$. One can show that at
least for some type of sentences and some theory $T$, the provability
implies the existence of a reduction.

Let the base theory be $S^1_2$ and the sentences have the form
$\forall x\exists^py.\phi(x,y)$, where $\phi$ is $\Sigma^b_1$. We will
show that provability of one sentence from the other implies the
existence of a polynomial reduction of one {\bf TFNP} problem to the
other. The following is a well-known fact (see~\cite{hanika}), but we
will still give a proof, because we want to argue that it can be
generalized to stronger theories.

\bpr
Suppose that $\N\models\forall x\exists^py.\phi(x,y)\wedge\forall u\exists^pv.\psi(u,v)$ and 
\bel{pr-red}
S^1_2\vdash \forall x\exists^py.\phi(x,y)\to\forall u\exists^pv.\psi(u,v),
\ee
where $\phi$ and $\psi$ define polynomial time computable relations. Then the {\bf TFNP} problem defined by $\psi$ is polynomially reducible to the {\bf TFNP} problem defined by $\psi$.
\epr
\bprf
This proposition is an immediate consequence of the following result (see~\cite{pudlak92}).
\bl
If $S^1_2\vdash \forall x\exists^p y\forall^p z.\alpha(x,y,z)$, 
%\marginpar{potrebujeme mit kvantory omezene? Ne, staci omezit z.}
where $\alpha$ is $\Pi^b_0$, then for a given $x$, one can compute $y$ such that $\forall^p z.\alpha(x,y,z)$ using a polynomial time oracle Turing machine with any oracle that, for a given $x$ and $y$ such that $\exists^p z.\neg\alpha(x,y,z)$ holds true, produces some $z$ such that $\neg\alpha(x,y,z)$ holds true.
\el
Write the implication in~(\ref{pr-red}) in the following prenex form
\[
\forall u\exists x\exists v\forall y(\phi(x,y)\to\psi(u,v)).
\]
By the lemma, there is a polynomial time Turing machine $M$ that computes $x$ and $v$ from a given $u$ using any oracle  that whenever $\exists y(\phi(x,y)\wedge\neg\psi(u,v))$ holds true produces a witness for $y$. We want to use an oracle that only produces witnesses for  $\exists y.\phi(x,y)$. Clearly, such an oracle suffices. If $M$ asks a query $(x,v)$ such that $\psi(u,v)$ is true, then we can stop, because we already have a witness for $\exists^pv.\psi(u,v)$. If no such query occurs during the computation of $M$, then the oracle must always produce a witness for $\exists y.\phi(x,y)$, hence  we get $x$ and $v$ such that $\forall y(\phi(x,y)\to\psi(u,v))$ is true, which is equivalent to $\exists y.\phi(x,y)\to\psi(u,v)$. But the antecedent is always true, so we have $\psi(u,v)$.
\eprf

If the base theory $T$ is stronger than $S^1_2$, we believe that
we nevertheless get some class of reductions that is probably stronger
than polynomial time computable reductions, but still somewhat restricted so that the classes of {\bf TFNP} equivalent with respect to these reductions do not completely collapse. These reductions should be defined
using the provably total search problems of~$T$. 
The idea is that the provably total polynomial search problems of $S^1_2$ are the problems solvable in polynomial time and this gives us reductions that are polynomial time computations with oracle queries to which we substitute solutions of the problem to which we are reducing the given problem. Similarly, if $\cal S$ is the class of provably total polynomial search problems of $T$, then provability in $T$ should give us reductions that are problems from $\cal S$ with oracle queries. A special case of this appeared in~\cite{BKT14} (not quite explicitly) where the theory was $T^1_2$ and the class of search problems was {\bf PLS}. Although it may be interesting to study such reductions in general, we believe that they would give the same conjectures if used instead of polynomial reductions.

%% What if there are complete problems (in {\bf TFNP}, disjoint pairs
%% etc.)? Answer: then the impossibility to prove the sentences may be
%% due to the complexity of \emph{reductions}.

\section{Conclusions and open problems}

In this paper we put forward the thesis that there exists a connection between the complexity of problems associated with first order sentences and their logical strength manifested as impossibility of proving them in weak theories. If we interpret this thesis in a broad sense, then the thesis is true; e.g., we cannot prove in a weak theory that some computation stops if the problem requires extremely long time to be solved. However, our argument here is that there may be such a connection already on the very low level, namely in the domain of problems solvable in nondeterministic polynomial time. Since the current state of research into such low complexity classes does not have the means to prove separations, we can only state and compare hypotheses about such a connection. 

There are two basic conjectures which have equivalent formulations and come in some flavors. The first one is about finite consistency statements and was proposed already a long time ago~\cite{KP}. The second one is more recent and concerns provably total polynomial search problems. We showed how they are related to some weaker statements and some stronger ones. Some of these statements had already been studied before. There are still many problems that need to be solved if we want to fully understand this topic; some are of a fundamental nature, some are more specific. Some problems have already been mentioned in previous sections. Below we briefly mention some more.

\ben
\item The main problem, mentioned in Subsection~\ref{sec-general}, is to find a general conjecture about incompleteness and computational complexity. The conjectures we studied in this paper should be special cases of it.
\item More specifically, propose a natural and plausible conjecture that implies the two main  Conjectures~{\sf CON} and~{\sf TFNP}, or prove that one of these conjectures implies the other, or show that their relativizations are independent. 
%\item Propose a natural and plausible conjecture stronger than Conjectures~\ref{c6} or~\ref{c8}.
\item Construct an oracle with respect to which Conjecture~{\sf DisjCoNP} is true. Construct oracles that show that relativized conjectures are different or show that they are equivalent for pairs of conjectures presented in this paper. Apparently the only separation that is known is a separation of Conjectures~{\sf CON} and~{\sf DisjNP}, see~\cite{GSSZ}.
\item In order to get more evidence for Conjecture~{\sf TFNP}, characterize provably total polynomial search problems in stronger systems of Bounded Arithmetic. The strongest theory for which a combinatorial characterization has been found is $V^1_2$, see~\cite{KNT11,BB14}.
\item Characterize more canonical pairs of propositional proof systems in order to get more evidence for Conjecture~{\sf DisjNP}. A combinatorial characterization of the canonical pair has only been found for Resolution, see~\cite{BPT14}. Characterize canonical pairs of some total polynomial search problems (as defined in this paper) in order to get some evidence for Conjecture~{\sf DisjCoNP}. Nothing is known in this direction.
\item We would also be interested in seeing connections between the non-existence of complete problems in some probabilistic classes and our main conjectures. K\"obler et al~\cite{KMT} proved that if {\it TAUT$_2$} (or {\it SAT$_2$}) have a p-optimal proof system, then {\bf BPP}, {\bf RP} and {\bf ZPP} have many-one complete problems. ({\it TAUT$_2$} and {\it SAT$_2$} are the sets of $\Pi_2$ and $\Sigma_2$ quantified Boolean tautologies.) But most researchers believe that these probabilistic classes do have complete problems, because they are in fact equal to {\bf P}.
\item Another important subject are \emph{proof-complexity generators} of Kraj\'{\i}\v{c}ek and Razborov (see~\cite{krajicek11,razborov15} and the references therein). One conjecture about proof-complexity generators states that they are hard for every proof system. It would be very interesting to find connections to our conjectures.
\een

\subsection*{Acknowledgment}
I would like to thank Pavel Hrube\v{s}, Emil Je\v{r}\'abek, Jan Kraj\'{\i}\v{c}ek and Neil Thapen for their useful comments on the draft of this paper. I am also indebted to an anonymous referee for pointing out many small errors and suggesting improvements of presentation.

%\section{Computational content of $\forall\bf P$ sentences}

\end{document}